\documentclass[12pt]{amsart}

\pdfoutput = 1

\usepackage[all]{xy}
\usepackage{amssymb}
\usepackage{eucal}

\usepackage[utf8]{inputenc}
\usepackage[T1]{fontenc}
\usepackage[backend=biber,
            isbn=false,
            doi=false,
            maxbibnames=5,
            giveninits=true,
            style=alphabetic,
            citestyle=alphabetic]{biblatex}
\usepackage{url}
\renewbibmacro{in:}{}
\bibliography{lectures.bib}

\usepackage[colorlinks=true]{hyperref}
\usepackage{cleveref}

\newcommand{\bA}{\mathbb{A}}
\newcommand{\bD}{\mathbb{D}}
\newcommand{\bE}{\mathbb{E}}
\newcommand{\bG}{\mathbb{G}}
\newcommand{\bL}{\mathbb{L}}
\newcommand{\bP}{\mathbb{P}}
\newcommand{\bT}{\mathbb{T}}

\newcommand{\B}{\mathrm{B}}
\newcommand{\C}{\mathrm{C}}
\renewcommand{\d}{\mathrm{d}}
\renewcommand{\H}{\mathrm{H}}
\renewcommand{\L}{\mathrm{L}}
\newcommand{\N}{\mathrm{N}}
\newcommand{\R}{\mathrm{R}}
\newcommand{\T}{\mathrm{T}}
\newcommand{\U}{\mathrm{U}}
\newcommand{\Z}{\mathrm{Z}}

\newcommand{\cA}{\mathcal{A}}
\newcommand{\cB}{\mathcal{B}}
\newcommand{\cC}{\mathcal{C}}
\newcommand{\cF}{\mathcal{F}}
\newcommand{\cG}{\mathcal{G}}
\newcommand{\cL}{\mathcal{L}}
\newcommand{\cM}{\mathcal{M}}
\newcommand{\cO}{\mathcal{O}}
\newcommand{\cS}{\mathcal{S}}

\newcommand{\g}{\mathfrak{g}}
\newcommand{\h}{\mathfrak{h}}

\newcommand{\alg}{\mathrm{Alg}}
\newcommand{\ialg}{\mathcal{A}\mathrm{lg}}
\newcommand{\bialg}{\mathrm{BiAlg}}
\newcommand{\calg}{\mathrm{CAlg}}
\newcommand{\icalg}{\mathcal{C}\mathrm{Alg}}
\newcommand{\ich}{\mathcal{C}\mathrm{h}}
\newcommand{\coalg}{\mathrm{CoAlg}}
\renewcommand{\mod}{\mathrm{Mod}}
\newcommand{\imod}{\mathcal{M}\mathrm{od}}
\newcommand{\ibimod}{\mathcal{B}\mathrm{iMod}}
\newcommand{\ilmod}{\mathcal{L}\mathrm{Mod}}
\newcommand{\irmod}{\mathcal{R}\mathrm{Mod}}

\newcommand{\ad}{\mathrm{ad}}

\newcommand{\Alt}{\mathrm{Alt}}
\newcommand{\Arr}{\mathrm{Arr}}
\newcommand{\Ass}{\mathrm{Ass}}
\newcommand{\Bun}{\mathrm{Bun}}
\newcommand{\CH}{\mathrm{CH}}
\newcommand{\coLie}{\mathrm{coLie}}
\newcommand{\Cois}{\mathrm{Cois}}
\newcommand{\Comm}{\mathrm{Comm}}
\newcommand{\Comp}{\mathrm{Comp}}
\newcommand{\DR}{\mathbf{DR}}
\newcommand{\dR}{\mathrm{dR}}
\newcommand{\ddr}{\mathrm{d}_{\dR}}
\newcommand{\fib}{\mathrm{fib}}
\newcommand{\forget}{\mathrm{forget}}
\newcommand{\Fun}{\mathrm{Fun}}
\newcommand{\hG}{\widehat{G}}
\newcommand{\Hom}{\mathrm{Hom}}
\newcommand{\id}{\mathrm{id}}
\newcommand{\IndCoh}{\mathrm{IndCoh}}
\newcommand{\Isot}{\mathrm{Isot}}
\newcommand{\Lagr}{\mathrm{Lagr}}
\newcommand{\Lie}{\mathrm{Lie}}
\newcommand{\Map}{\mathrm{Map}}
\newcommand{\MC}{\mathrm{MC}}
\newcommand{\uMC}{\underline{\mathrm{MC}}}
\newcommand{\Obs}{\mathrm{Obs}}
\newcommand{\obs}{\mathrm{obs}}
\newcommand{\opp}{\mathrm{opp}}
\newcommand{\Perf}{\mathrm{Perf}}
\newcommand{\PGL}{\mathrm{PGL}}
\newcommand{\Pois}{\mathrm{Pois}}
\newcommand{\Pol}{\mathrm{Pol}}
\newcommand{\pt}{\mathrm{pt}}
\newcommand{\QCoh}{\mathrm{QCoh}}
\newcommand{\Sym}{\mathrm{Sym}}
\newcommand{\Symp}{\mathrm{Symp}}
\newcommand{\Tot}{\mathrm{Tot}}
\newcommand{\triv}{\mathrm{triv}}

\DeclareMathOperator{\colim}{colim}
\DeclareMathOperator{\Rep}{Rep}
\DeclareMathOperator{\Spec}{Spec}

\newcommand{\adj}[2]{
\xymatrix{
#1 \ar@<.5ex>[r] & #2 \ar@<.5ex>[l]
}
}

\newcommand{\cosimp}[2]{
\xymatrix{
#1 \ar@<.5ex>[r] \ar@<-.5ex>[r] & #2 \ar@<.8ex>[r] \ar[r] \ar@<-.8ex>[r] & \ldots
}
}

\newcommand{\simp}[2]{
\xymatrix{
#1 & #2 \ar@<.5ex>[l] \ar@<-.5ex>[l] & \ldots \ar@<.8ex>[l] \ar[l] \ar@<-.8ex>[l]
}
}

\newcommand{\defterm}[1]{\textbf{\emph{#1}}}

\newtheorem{thm}{Theorem}[section]
\newtheorem{prop}[thm]{Proposition}
\newtheorem{cor}[thm]{Corollary}
\newtheorem{lm}[thm]{Lemma}
\newtheorem{conjecture}[thm]{Conjecture}

\theoremstyle{definition}
\newtheorem{defn}[thm]{Definition}
\newtheorem{notation}[thm]{Notation}

\theoremstyle{remark}
\newtheorem{remark}[thm]{Remark}
\newtheorem{example}[thm]{Example}

\begin{document}
\title{Lectures on shifted Poisson geometry}
\author{Pavel Safronov}
\address{Institut f\"{u}r Mathematik, Winterthurerstrasse 190, 8057 Z\"{u}rich, Switzerland}
\email{pavel.safronov@math.uzh.ch}
\begin{abstract}
These are expanded notes from lectures given at the \'{E}tats de la Recherche workshop on ``Derived algebraic geometry and interactions''. These notes serve as an introduction to the emerging theory of Poisson structures on derived stacks.
\end{abstract}
\maketitle

\section*{Introduction}

\subsection*{Higher Poisson structures}

Poisson and symplectic structures go back to the works of Lagrange and Poisson on classical mechanics. Classical mechanical systems describe individual particles and as such are examples of 1-dimensional classical field theories.

A precursor to the theory of higher Poisson structures is the work \cite{BV} of Batalin and Vilkovisky on field theories with symmetries. Their work shows how to endow the (derived) critical locus of an action functional with a $(-1)$-shifted symplectic structure such that an introduction of a symmetry corresponds to a $(-1)$-shifted symplectic reduction. The first explicit occurrence of higher symplectic geometry in relation to classical field theories is the work \cite{AKSZ} describing how to obtain action functionals of $n$-dimensional topological field theories from NQ manifolds equipped with a symplectic structure of degree $(n-1)$ (a smooth analog of the notion of an $(n-1)$-shifted symplectic stack). We can informally summarize the AKSZ construction as follows:
\begin{itemize}
\item All classical topological field theories of dimension $n$ arise from $(n-1)$-shifted symplectic stacks.
\end{itemize}

For instance, the case of classical mechanics corresponds to $n=1$ in which case we have $0$-shifted (i.e. ordinary) symplectic manifolds. The work of Batalin and Vilkovisky corresponds to $n=0$ in which case we have $(-1)$-shifted symplectic stacks.

In a different direction, the theory of multiplicative Poisson structures (Poisson-Lie groups and Poisson groupoids) gave rise to the notions of Courant algebroids and Dirac structures, see \cite{LWX}. In turn, Courant algebroids and Dirac structures were later interpreted in terms of higher symplectic geometry by Roytenberg \cite{Roy} and \v{S}evera \cite{Sev}.

\subsection*{Deformation quantization}

Another motivation to study higher Poisson structures is the theory of deformation quantization. In quantum mechanics observables form an associative algebra which is supposed to be a deformation (with deformation parameter $\hbar$, the Planck's constant) of the Poisson algebra of functions on the classical phase space. Once we generalize commutative algebras to commutative dg algebras, it is natural to consider Poisson algebras with the bracket of a nonzero cohomological degree. These are known as $\bP_n$-algebras so that a $\bP_1$-algebra is the same as a dg Poisson algebra.

Let us now consider an $n$-dimensional topological field theory. As we have mentioned, such a field theory on the classical level is determined by an $(n-1)$-shifted symplectic stack whose algebra of functions has a $\bP_n$-structure. The operator product expansion endows the cohomology of local quantum observables with a commutative product. It turns out that on the chain level the product is not coherently commutative but instead forms what is known as an $\bE_n$-algebra. We refer to \cite{CG} for an explanation of related ideas. For instance, an $\bE_\infty$-algebra is a commutative dg algebra and an $\bE_1$-algebra is an associative algebra. Similar to the previous case, it is expected that the $\bE_n$-algebras of local quantum observables are deformation quantizations of the $\bP_n$-algebras of functions on the clasical phase space described by the $(n-1)$-shifted symplectic stack.

An example of higher deformation quantization is given by the theory of quantum groups which can be considered as deformation quantizations of the $2$-shifted symplectic stack $\B G$, the classifying stack of a reductive group $G$. 

The theory of higher deformation quantization is still in its infancy, so we will not discuss it in the main body of the text. Let us just mention several recent articles that deal with this topic: \cite{ToICM}, \cite[Section 3.5]{CPTVV}, \cite[Section 5]{MS2}, \cite{Pri2}, \cite{Pri3} and \cite{Pri4}.

\subsection*{Implications for classical geometry}

So far we have explained why one is interested in generalizations of Poisson geometry to the derived setting. However, the derived perspective can also be useful for questions involving only classical geometric objects.

In Donaldson--Thomas theory one considers moduli spaces $\cM_{st}(X)$ of stable sheaves on a Calabi--Yau threefold $X$. The Donaldson--Thomas invariants can be computed in terms of a symmetric obstruction theory on $\cM_{st}(X)$, see e.g. \cite{BF} and \cite{Behr2}. The moduli spaces $\cM_{st}(X)$ are open substacks in the classical truncation of the derived moduli stack $\Perf(X)$ of perfect complexes on $X$. In turn, the symmetric obstruction theory on $\cM_{st}(X)$ is a shadow of the $(-1)$-shifted symplectic structure on $\Perf(X)$, see \cite[Section 3.2]{PTVV}. Even though the symmetric obstruction theory is enough to define ordinary Donaldson--Thomas invariants, for their categorification one needs to consider the full $(-1)$-shifted symplectic structure on $\Perf(X)$, see \cite{BBBBJ}. Moreover, the existence of the $(-1)$-shifted symplectic structure on $\Perf(X)$ implies that $\cM_{st}(X)$ can be covered by charts given by a critical locus of a function.

Another fruitful application is the construction of Poisson structures on several moduli spaces $\cM$. Classically it proceeds as follows. First, one restricts to an open subspace $\cM^{sm}\subset \cM$ consisting of smooth and non-stacky points. Second, it is usually easy to define a bivector on $\cM^{sm}$. The final step involves a nontrivial computation to show that the bivector satisfies the Jacobi identity.

Suppose one can realize $\cM$ as a classical truncation of a derived mapping stack. Then it inherits shifted symplectic and shifted Lagrangian structures by the AKSZ construction. Using a relation between shifted symplectic and shifted Poisson structures one can endow $\cM$ with a 0-shifted Poisson structure which then gives an ordinary Poisson structure on the smooth locus $\cM^{sm}$. We refer to \cite[Section 3.1]{PTVV} which treats the Goldman Poisson structure \cite{Go} on the character variety and the symplectic structure on the moduli of bundles on a K3 surface \cite{Muk} and to \cref{sect:modulibundles} for the case of the Feigin--Odesskii Poisson structure \cite{FO}.

\subsection*{Conventions}

\begin{itemize}
\item Throughout the notes we will freely use the language of $\infty$-categories. We refer to \cite{HTT} and \cite{HA} as foundational texts which use quasi-categories as models for $\infty$-categories and to \cite{Gro} as an introduction to the theory. This language is indispensable in dealing with higher homotopical structures and descent questions.

\item We will work in the framework of derived algebraic geometry, see \cite{HAGII} for details and \cite{Cal2}, \cite{To} and \cite{ToStacks} for an introduction. We work over a base field $k$ of characteristic zero.
\end{itemize}

In this text we mainly review the papers \cite{CPTVV}, \cite{MS1}, \cite{MS2} and \cite{Pri1} on shifted Poisson geometry without many technical details to simplify the exposition. We refer to \cite{PV} which treats certain topics omitted in this review.

\subsection*{Acknowledgements}

The author would like to thank A. Brochier, D. Calaque, G. Ginot, V. Melani, T. Pantev, B. Pym, N. Rozenblyum and B. To\"{e}n for many conversations about shifted Poisson structures. In particular, the author thanks B. Pym for reading a draft of the notes and making several useful suggestions. While writing this overview, the author was supported by the NCCR SwissMAP grant of the Swiss National Science Foundation.

\section{Shifted Poisson algebras}

In the first lecture we review what Poisson and $\bP_n$-algebras are, explain how to relate $\bP_n$-algebras for different $n$ using Poisson additivity and discuss the definition of an $n$-shifted Poisson structure on a derived Artin stack.

\subsection{Classical definitions}

Let $X$ be a smooth scheme over a field $k$ and $\cO_X$ the structure sheaf.

\begin{defn}
A \defterm{Poisson structure} on $X$ is the data of a Lie bracket
\[\{-, -\}\colon \cO_X\otimes_k \cO_X\rightarrow \cO_X\]
satisfying the Leibniz rule $\{f, gh\} = \{f, g\}h + \{f, h\} g$ for any $f,g,h\in\cO_X$.
\end{defn}

If $X$ is affine, a Poisson structure on $X$ in the above sense can be equivalently stated in terms of a Lie bracket on the algebra of global functions $\cO(X)$.

\begin{example}
Suppose $Y$ is a smooth affine scheme and consider $X = \T^* Y$. Then
\[\cO(X) \cong \Gamma(Y, \Sym(\T_Y)).\]
It is generated as a commutative algebra by $\cO(Y)$ and $\Gamma(Y, \T_Y)$. By the Leibniz rule the Poisson bracket is uniquely determined on generators and we let
\begin{align*}
\{f, g\} &= 0 ,\qquad f,g\in\cO(Y)\\
\{v, f\} &= v.f,\qquad f\in \cO(Y),\ v\in\Gamma(Y, \T_Y) \\
\{v, w\} &= [v, w],\qquad v,w\in\Gamma(Y, \T_Y)
\end{align*}
\label{ex:cotangentbundle}
\end{example}

Here is a slight reformulation of the above definition. A Poisson structure on $X$ gives rise to an antisymmetric biderivation on $\cO_X$, i.e. an element
\[\pi\in\Gamma(X, \wedge^2\T_X).\]

More generally, define
\[\Pol(X, 0) = \Gamma(X, \Sym(\T_X[-1])),\]
the graded commutative algebra of polyvector fields. We call the $\Sym$ grading the \defterm{weight} grading, e.g. vector fields have weight $1$ and bivectors have weight $2$. The algebra $\Pol(X, 0)$ has a Lie bracket of cohomological degree $-1$ called the \defterm{Schouten bracket} which is given by formulas similar to those of \cref{ex:cotangentbundle}.

\begin{prop}
A Poisson structure on $X$ is equivalent to the data of an element $\pi\in\Pol(X, 0)$ of weight $2$ satisfying
\[[\pi, \pi] = 0.\]
\label{prop:poissonjacobi}
\end{prop}

One has yet another equivalent definition of a Poisson structure given as follows.

\begin{defn}
A \defterm{symplectic Lie algebroid} on $X$ is the data of a Lie algebroid $\cL$ over $X$ equipped with an isomorphism $\alpha\colon \cL\xrightarrow{\sim} \T^*_X$ satisfying
\begin{align}
\iota_{a(l)}\alpha(l) &= 0 \label{eq:sympliealgd1} \\
\alpha([l_1, l_2]) &= L_{a(l_1)} \alpha(l_2) - \iota_{a(l_2)}(\ddr \alpha(l_1)). \label{eq:sympliealgd2}
\end{align}
\end{defn}

\begin{prop}
A Poisson structure on $X$ is equivalent to the data of a symplectic Lie algebrod $\cL$ on $X$.
\label{prop:symplecticliealgd}
\end{prop}
\begin{proof}
Suppose $X$ has a Poisson structure $\pi$. Let $\cL = \T^*_X$. Define the anchor $a\colon \cL\rightarrow \T_X$ to be $l\mapsto \iota_l(\pi)$. Given two differential forms $l_1, l_2\in\cL$ we define the bracket (the so-called \emph{Koszul bracket}) to be
\[[l_1, l_2] = L_{a(l_1)} l_2 - L_{a(l_2)} l_1 - \ddr \pi(l_1, l_2).\]
We leave it to the reader to check that $\cL$ is a symplectic Lie algebroid.

Conversely, given a symplectic Lie algebroid $\cL$, the composite
\[\T^*_X\cong \cL\xrightarrow{a} \T_X\]
defines an element $\pi\in\Gamma(X, \T_X\otimes \T_X)$. The equation \eqref{eq:sympliealgd1} is equivalent to the antisymmetry of $\pi$ and equation \eqref{eq:sympliealgd2} is equivalent to the Jacobi identity for $\pi$.
\end{proof}

To summarize, we have the following three definitions of a Poisson structure on a smooth scheme $X$:
\begin{enumerate}
\item As an enhancement of the structure sheaf of $X$ to a sheaf of Poisson algebras. The derived version is given by \cref{defn:nshiftedPoisson}.

\item As a bivector satisfying the Jacobi identity. The derived version is given by \cref{prop:poissonpolyvectors}.

\item As a symplectic Lie algebroid. We refer to \cref{conj:sympliealgd} for its derived version.
\end{enumerate}

\subsection{Shifted Poisson algebras}

We are now going to explain what \emph{higher} Poisson structures are. We will discuss only the affine case deferring to \cref{sect:globalization} for the case of derived stacks.

One can generalize the notion of a Poisson algebra to the dg setting by considering a commutative dg algebra equipped with a bracket satisfying the Leibniz rule and graded antisymmetry. Since we have a grading, it is also natural to consider a generalization where we allow the degree of the bracket to be arbitrary.

\begin{defn}
A \defterm{$\bP_n$-algebra} is a commutative dg algebra $A$ over $k$ equipped with a bracket
\[\{-,-\}\colon A\otimes_k A\rightarrow A[1-n]\]
which gives a Lie bracket on $A[n-1]$ satisfying
\[\{f, gh\} = \{f, g\}h + (-1)^{|g||h|}\{f, h\} g,\qquad \forall f,g,h\in A.\]
\end{defn}

Our next goal is to study all ways of endowing a commutative dg algebra $A$ with a $\bP_n$-bracket. We can consider the set of such lifts which we temporarily denote by $\widetilde{\Pois}(A, n-1)$. It is clear that an \emph{isomorphism} of commutative dg algebras $A_1\cong A_2$ induces an isomorphism of sets $\widetilde{\Pois}(A_1, n-1)\cong \widetilde{\Pois}(A_2, n-1)$.

In derived algebraic geometry we are also interested in the behavior of properties under \emph{quasi-isomorphisms} $A_1\rightarrow A_2$ and in this case there is no obvious relation between the sets $\widetilde{\Pois}(A_1, n-1)$ and $\widetilde{\Pois}(A_2, n-1)$. Nevertheless, by the homotopy transfer theorem \cite[Section 10.3]{LV} given a $\bP_n$-algebra structure on $A_1$ we can transfer it to a \emph{homotopy} $\bP_n$-algebra structure on $A_2$ which is uniquely defined modulo $\infty$-quasi-isomorphisms. Thus, we should replace $\widetilde{\Pois}(A, n-1)$ by the set of $\infty$-quasi-isomorphism classes of homotopy $\bP_n$-structures on $A$ which are compatible with the given commutative algebra structure on $A$. Instead of passing to equivalence classes we will construct an $\infty$-groupoid $\Pois(A, n-1)$ whose objects are homotopy $\bP_n$-structures on $A$ and (higher) morphisms are given by (higher) homotopies. Such an $\infty$-groupoid can be constructed quite explicitly by hand; instead, we will define it in a more concise way using the $\infty$-category of $\bP_n$-algebras.

We will now switch to the language of dg operads for which the reader is referred to \cite{LV}. The above definition of a $\bP_n$-algebra gives a quadratic dg operad $\bP_n$ generated by two binary operations: multiplication of degree $0$ and bracket of degree $1-n$, so that a $\bP_n$-algebra is an algebra over the dg operad $\bP_n$.

\begin{notation}
Let $\cO$ be a dg operad such as the associative operad $\Ass$ or the commutative operad $\Comm$.
\begin{itemize}
\item We denote by $\alg_{\cO}$ the category of $\cO$-algebras in complexes. We also introduce the notations
\[\alg = \alg_{\Ass},\qquad \calg = \alg_{\Comm}\]
for the categories of associative and commutative dg algebras.

\item Let $W\subset \alg_{\cO}$ be the wide subcategory consisting of quasi-isomorphisms, i.e. those morphisms of $\cO$-algebras $A_1\rightarrow A_2$ which induce isomorphisms $\H^\bullet(A_1)\rightarrow \H^\bullet(A_2)$.
\end{itemize}
\end{notation}

\begin{defn}
The \defterm{$\infty$-category of $\cO$-algebras} is the localization
\[\ialg_{\cO} = \alg_{\cO}[W^{-1}].\]
\end{defn}

Similarly, we have
\[\ialg = \ialg_{\Ass},\qquad \icalg = \ialg_{\Comm}.\]

We have a functor $\alg_{\bP_n}\rightarrow \calg$ given by forgetting the bracket which after localization induces a functor of $\infty$-categories
\[\ialg_{\bP_n}\rightarrow \icalg.\]

Given an $\infty$-category $\cC$ we denote by $\cC^\sim$ the underlying $\infty$-groupoid of objects where we throw away non-invertible morphisms.

\begin{defn}
Let $A$ be a commutative dg algebra. The \defterm{$\infty$-groupoid $\Pois(A, n)$ of $n$-shifted Poisson structures on $A$} is the fiber of the forgetful map
\[\ialg_{\bP_{n+1}}^{\sim}\longrightarrow \icalg^{\sim}\]
at $A\in\icalg^{\sim}$.
\end{defn}

In other words, an object of $\Pois(A, n)$ is given by a pair of a $\bP_{n+1}$-algebra $\tilde{A}$ and a quasi-isomorphism of commutative dg algebras $\tilde{A}\xrightarrow{\sim} A$.

\begin{remark}
We will use the terms ``$\infty$-groupoid'' and ``space'' interchangeably as both will refer to objects of equivalent $\infty$-categories. In the model of $\infty$-categories as simplicial categories, we can take the relevant $\infty$-category $\cS$ to be that of Kan complexes.
\end{remark}

\begin{remark}
The same definition of the $\infty$-groupoid of $n$-shifted Poisson structures can be given for a commutative algebra in a general enough symmetric monoidal dg category $\cM$ (see \cite[Section 1.1]{CPTVV} for precise assumptions). This will be used to extend this notion to the non-affine setting in \cref{sect:globalization}.
\label{remark:generalmodelcategory}
\end{remark}

\subsection{Additivity}
\label{sect:additivity}

A natural question is whether $\bP_n$-algebras for different $n$ are related in some way. To explain the relation, let us recall the following result.

Let $\bE_n$ the topological operad of little $n$-disks. That is, the space of $m$-ary operations $\bE_n(m)$ is the space of rectilinear embeddings of $m$ $n$-dimensional disks into a fixed $n$-dimensional disk. For instance, in the case $n=1$ $\bE_1(m)$ is the space of embeddings of $m$ intervals $[0, 1]$ into $[0, 1]$. Such a space retracts to $S_m$ and hence the operad $\bE_1$ is equivalent to the associative operad. Thus, an $\bE_1$-algebra in complexes is simply an associative dg algebra. In the opposite limit the space $\bE_\infty(m)$ is weakly contractible, so $\bE_\infty$ is equivalent to the commutative operad and an $\bE_\infty$-algebra in complexes is a commutative dg algebra.

Given any symmetric monoidal $\infty$-category $\cC$ we can consider the $\infty$-category $\ialg_{\bE_n}(\cC)$ of $\bE_n$-algebras in $\cC$. The following is \cite[Theorem 5.1.2.2]{HA}.

\begin{thm}[Dunn--Lurie]
One has an equivalence of symmetric monoidal $\infty$-categories
\[\ialg_{\bE_{n+m}}(\cC)\cong \ialg_{\bE_n}(\ialg_{\bE_m}(\cC)).\]
\label{thm:dunnlurieadditivity}
\end{thm}

Let $R$ be a commutative dg algebra over $k$ and denote by $\imod_R$ the $\infty$-category of dg $R$-modules defined as
\[\imod_R = \mod_R[W^{-1}].\]
We can also consider the $\infty$-category of $\bP_n$-algebras in $\imod_R$ by defining
\[\ialg_{\bP_n}(\imod_R) = \alg_{\bP_n}(\mod_R)[W^{-1}],\]
where $\mod_R$ is the dg category of $R$-modules. The following is \cite[Theorem 2.22]{Sa2} and was previously proved independently by Rozenblyum (unpublished).

\begin{thm}
One has an equivalence of symmetric monoidal $\infty$-categories
\[\ialg_{\bP_{n+m}}(\imod_R)\cong \ialg_{\bE_n}(\ialg_{\bP_m}(\imod_R)).\]
\label{thm:poissonadditivity}
\end{thm}

\begin{remark}
The statement in \cite{Sa2} was proved in the case $R=k$, but the proof can be generalized to any commutative dg algebra over $k$.
\end{remark}

Let us sketch the construction of the equivalence in \cref{thm:poissonadditivity}. The reader unfamiliar with the theory of operads can safely skip to the corollaries of the construction (\cref{prop:poissonlinearization} and \cref{prop:poissonforget2}).

Given a Lie algebra $\g$ and an element $x\in\g$ we denote by $\ad_x\colon \g\rightarrow \g$ the adjoint action $[x, -]$.

\begin{defn}
An \defterm{$n$-shifted Lie bialgebra} is a dg Lie algebra $\g$ equipped with a degree $-n$ Lie cobracket $\delta\colon \g\otimes \g\rightarrow \g[-n]$ satisfying the relation
\[\delta([x, y]) = (\ad_x\otimes \id + \id\otimes \ad_x)\delta(y) - (-1)^{n + |x||y|}(\ad_y\otimes\id + \id\otimes\ad_y)\delta(x)\]
for any $x,y\in\g$.
\end{defn}

\begin{defn}
A dg Lie coalgebra $\g$ is \defterm{conilpotent} if for every $x\in\g$ there is an $N$ so that the $N$-fold application of $\delta$ to $x$ gives zero.
\end{defn}

Denote by $\bialg_{\Lie_n}$ the category of $n$-shifted Lie bialgebras conilpotent as Lie coalgebras. Then one has a bar-cobar adjunction
\[\adj{\Omega\colon \bialg_{\Lie_{n-1}}}{\alg_{\bP_{n+1}}\colon \B}\]
constructed as follows. Suppose $\g$ is an $(n-1)$-shifted Lie bialgebra. Then $A=\Sym(\g[-n])$ is endowed with the Chevalley--Eilenberg differential coming from the Lie coalgebra structure on $\g$. Moreover, the Lie bracket on $\g$ induces a $\bP_{n+1}$-structure on $A$. Given an $R$-module $V$ we denote by $\coLie(V)$ the cofree conilpotent Lie coalgebra cogenerated by $V$.

If $A$ is a $\bP_{n+1}$-algebra we can endow $\coLie(A[1])[n-1]$ with the Harrison differential and an $(n-1)$-shifted Lie bialgebra structure, see \cite[Section 2.4]{Sa2} for details. We call the functor $\Omega$ the cobar construction and $\B$ the bar construction.

We also have another bar-cobar adjunction
\[\adj{\coalg_{\bP_{n+1}}}{\alg_{\bP_{n+1}}}\]
between $\bP_{n+1}$-coalgebras and $\bP_{n+1}$-algebras constructed in a similar way.

Both categories $\bialg_{\Lie_{n-1}}$ and $\coalg_{\bP_{n+1}}$ have a class of weak equivalences which are those morphisms which become quasi-isomorphisms after applying $\Omega$.

\begin{lm}
The bar-cobar adjunctions
\[\adj{\Omega\colon \bialg_{\Lie_{n-1}}}{\alg_{\bP_{n+1}}\colon \B}\]
and
\[\adj{\coalg_{\bP_{n+1}}}{\alg_{\bP_{n+1}}}\]
induce equivalences on the underlying $\infty$-categories.
\end{lm}

Recall that given a Lie algebra $\g$ its universal enveloping algebra $\U(\g)$ is a cocommutative bialgebra. Moreover, we have an equivalence of categories
\[\alg_{\Lie}\xrightarrow{\U} \alg(\coalg_{\Comm}),\]
where $\alg(-)$ is the category of associative algebras and $\alg(\coalg_{\Comm})$ is the category of conilpotent cocommutative bialgebras. Similarly, if $\g$ is an $(n-1)$-shifted Lie bialgebra, $\U(\g)$ inherits a natural $\bP_n$-cobracket which induces an equivalence of categories
\[\bialg_{\Lie_{n-1}}\xrightarrow{\U} \alg(\coalg_{\bP_n}).\]

The composite
\[\alg_{\bP_{n+1}} \xrightarrow{\B} \bialg_{\Lie_{n-1}} \xrightarrow{\U} \alg(\coalg_{\bP_n})\]
after passing to underlying $\infty$-categories induces the required functor
\begin{equation}
\ialg_{\bP_{n+1}}\longrightarrow \ialg(\ialg_{\bP_n}).
\label{eq:poissonadditivity}
\end{equation}

Note that the category $\alg_{\bP_{n+1}}$ has an involution given by flipping the sign of the bracket. Under the bar complex $\B$ it goes to the same involution of $\bialg_{\Lie_{n-1}}$. Applying the universal enveloping algebra $\U$ it goes to the involution of $\alg(\coalg_{\bP_n})$ given by passing to the opposite algebra.

\begin{remark}
The $\infty$-category $\ialg(\ialg_{\bP_n})$ has two involutions: passing to the opposite associative algebra and passing to the opposite $\bP_n$-algebra. The antipode of $\U(\g)$ gives a natural isomorphism of the two involutions of $\ialg(\ialg_{\bP_n})$.
\end{remark}

\begin{prop}
The Poisson additivity equivalence \eqref{eq:poissonadditivity} intertwines the involution of $\ialg_{\bP_{n+1}}$ given by flipping the sign of the bracket and the involution of $\ialg(\ialg_{\bP_n})$ given by passing to the opposite algebra.
\label{prop:oppositepoisson}
\end{prop}

Let us present some corollaries of this construction of the additivity functor. Observe that the additivity functor gives rise to a forgetful functor
\[\forget_{n+1}^n\colon \ialg_{\bP_{n+1}}\xrightarrow{\sim} \ialg(\ialg_{\bP_n})\longrightarrow \ialg_{\bP_n}\]
which has the following properties.

\begin{prop}
Let $\g$ be an $(n-1)$-shifted Lie bialgebra and $A=\Omega\g$ the associated $\bP_{n+1}$-algebra. Denote by $A_0\in\ialg_{\bP_n}$ the same commutative dg algebra with the trivial bracket. Then an equivalence of $\bP_n$-algebras $\forget_{n+1}^n(A)\cong A_0$ is the same as an equivalence of $\bP_n$-coalgebras $\U(\g)\cong \Sym(\g)$.
\label{prop:poissonlinearization}
\end{prop}

Observe that $\U(\g)$ is the algebra of distributions on the formal Poisson-Lie group $G$ integrating $\g$ and $\Sym(\g)$ is the algebra of distributions on the completion $\widehat{\g}_0$. Thus, the forgetful functor $\forget_{n+1}^n$ is closely related to the theory of formal linearizations of Poisson-Lie groups; we refer to \cite[Proposition 2.17]{Sa3} for more details.

Even though the forgetul functor $\forget_n^{n-1}\colon \ialg_{\bP_n}\rightarrow \ialg_{\bP_{n-1}}$ is nontrivial, the iterated forgetful functors $\forget_n^{n-m}\colon \ialg_{\bP_n}\rightarrow \ialg_{\bP_{n-m}}$ for $m\geq 2$ are essentially trivial, i.e. the corresponding $\bP_{n-m}$-algebra is automatically commutative.

\begin{prop}
The choice of a Drinfeld associator gives a homotopy commutativity data in the diagram
\[
\xymatrix{
\ialg_{\bP_{n+1}} \ar_{\forget}[dr] \ar^-{\forget_{n+1}^{n-1}}[rr] && \ialg_{\bP_{n-1}} \\
& \icalg \ar_{\triv}[ur] &
}
\]
\label{prop:poissonforget2}
\end{prop}
\begin{proof}
Let $\g$ be an $(n-1)$-shifted Lie bialgebra and $A=\Omega\g$ the corresponding $\bP_{n+1}$-algebra. The functor $\ialg_{\bP_{n+1}}\rightarrow \ialg_{\bE_2}(\ialg_{\bP_{n-1}})$ is given by
\[A\mapsto \U_{\bE_2}(\g),\]
where $\U_{\bE_2}(-)$ is the $\bE_2$ enveloping algebra.

Given a Drinfeld associator, by the results of \cite{Ta} we obtain an equivalence of Hopf operads $\C_\bullet(\bE_2)\cong \bP_2$ and hence an equivalence of universal enveloping functors $\U_{\bE_2}(\g)\cong \U_{\bP_2}(\g)$. The latter object can be identified with
\[\U_{\bP_2}(\g)\cong \Sym(\g[-1])\]
as $\bP_{n-1}$-coalgebras. But $\Sym(\g[-1])\in\coalg_{\bP_{n-1}}$ corresponds under Koszul duality to the $\bP_{n-1}$-algebra $A$ with the trivial bracket.
\end{proof}

\subsection{Polyvectors}

\label{sect:polyvectors}

Now we are going to give a way to compute the $\infty$-groupoid $\Pois(A, n)$ of $n$-shifted Poisson structures on $A$. Let us recall from \cite[Section 5.3]{HA} the notion of a centralizer in a symmetric monoidal $\infty$-category $\cC$.

\begin{defn}
Let $f\colon A\rightarrow B$ a morphism in $\cC$. The \defterm{centralizer of $f$} is the universal object $\Z(f)\in\cC$ equipped with the following data:
\begin{enumerate}
\item A morphism $1\rightarrow \Z(f)$,

\item A morphism $\Z(f)\otimes A\rightarrow B$,

\item A commutativity data in the diagram
\[
\xymatrix{
& \Z(f)\otimes A \ar[dr] & \\
A \ar[ur] \ar^{f}[rr] && B
}
\]
\end{enumerate}
\end{defn}

\begin{defn}
For an object $A\in\cC$ its \defterm{center $\Z(A)$} is the centralizer $\Z(\id\colon A\rightarrow A)$.
\label{defn:center}
\end{defn}

Note that a centralizer may not exist in general; we will say that $\cC$ \defterm{admits centralizers} if any morphism $f$ in $\cC$ has a centralizer $\Z(f)\in\cC$.

\begin{example}
Suppose $\cC$ is a closed symmetric monoidal $\infty$-category. Then it admits centralizers given by $\Z(f)=\Hom_\cC(A, B)$, the internal $\Hom$.
\end{example}

\begin{example}
Consider the symmetric monoidal $\infty$-category $\ialg(\ich)$ of dg algebras. The center of a dg algebra $A\in\ialg(\ich)$ is an object $\Z(A)\in\ialg(\ialg(\ich))$ which by the Dunn--Lurie additivity \cref{thm:dunnlurieadditivity} is the same as an $\bE_2$-algebra. One can show that this $\bE_2$-algebra is equivalent to the Hochschild complex $\CH^\bullet(A, A)$ (see e.g. \cite[Section 7.2]{Gin}). When $A$ is concentrated in degree $0$, the zeroth cohomology of the Hochschild complex $\CH^\bullet(A, A)$ is isomorphic to the ordinary center of $A$ as an associative algebra. Thus, we see that $\Z(A)$ is a derived generalization of this notion.
\end{example}

\begin{example}
Let $X=\Spec A$ be a smooth affine Poisson scheme. Then $\Z(A)$ is quasi-isomorphic to the complex $\Gamma(X, \Sym(\T_X[-1]))$ equipped with the differential $[\pi_X, -]$, i.e. the Poisson cohomology of $X$. Its zeroth cohomology coincides with the space of Casimir functions of $A$, i.e. elements $f\in A$ such that $\{f, g\}=0$ for every $g\in A$.
\end{example}

Suppose $\cC$ admits centralizers and consider a sequence of morphisms $A\xrightarrow{f} B\xrightarrow{g} C$ in $\cC$. From the universal property we obtain an action morphism
\[\Z(f)\otimes \Z(g)\rightarrow \Z(g\circ f).\]

In particular, we can upgrade $\Z(A)=\Z(\id)$ to an object of $\ialg(\cC)$. We have the following result given by \cite[Theorem 5.3.1.14]{HA}:

\begin{thm}
Suppose $\cO$ is a coherent $\infty$-operad and $\cC$ is a symmetric monoidal $\infty$-category whose tensor product preserves colimits. Then $\ialg_{\cO}(\cC)$ admits centralizers.
\end{thm}

At the moment the above statement is not known for dg operads (see, however, the recent preprint \cite{CH} which discusses the appropriate notion of a linear $\infty$-operad). Nevertheless, we have the following statement whose proof will appear in a future paper.

\begin{thm}
Let $R$ be a commutative dg algebra over $k$. Then the symmetric monoidal $\infty$-category of $R$-linear $\bP_n$-algebras $\ialg_{\bP_n}(\imod_R)$ admits centralizers.
\end{thm}

We are now going to construct the algebra of $n$-shifted polyvector fields.

\begin{defn}
A \defterm{graded dg Lie algebra} is a dg Lie algebra $\g$ equipped with a weight grading $\g=\bigoplus_n \g_n$ such that the Lie bracket has weight $-1$.
\end{defn}

Similarly, one can define a graded $\bP_n$-algebra to be a $\bP_n$-algebra equipped with an extra weight grading such that the Lie bracket has weight $-1$. Thus, both $\Lie$ and $\bP_n$ lift to operads in graded complexes. We denote by $\ialg_{\bP_n}^{gr}$ the $\infty$-category of graded $\bP_n$-algebras. We denote by
\[\triv_{\bP_n}\colon \icalg\longrightarrow \ialg_{\bP_n}^{gr}\]
the functor which sends a commutative dg algebra $A$ to the graded $\bP_n$-algebra concentrated in weight $0$ with the zero bracket.

\begin{defn}
Let $A\in\icalg$ be a commutative dg algebra. The \defterm{algebra of $n$-shifted polyvector fields on $A$} is
\[\Pol(A, n) = \Z(\triv_{\bP_{n+1}}(A))\in\ialg(\ialg_{\bP_{n+1}}^{gr}).\]
\end{defn}

In particular, using \cref{thm:poissonadditivity} we see that $\Pol(A, n)\in\ialg_{\bP_{n+2}}^{gr}$. The following statement gives a control over $\Pol(A, n)$ as a graded commutative dg algebra.

\begin{prop}
Let $A$ be a commutative dg algebra and $\bL_A\in\imod_A$ its cotangent complex. Then we have an equivalence of graded commutative dg algebras
\[\Pol(A, n)\cong \Hom_A(\Sym_A(\bL_A[n+1]), A).\]
\end{prop}

One can think of the Lie bracket on $\Pol(A, n)$ as a generalization of the Schouten bracket of polyvector fields. A weight 2 element of $\Pol(A, n)$ is a bivector on $A$ and by \cref{prop:poissonjacobi} one can describe $\bP_{n+1}$-brackets on $A$ in terms of bivectors $\pi$ satisfying $[\pi, \pi] = 0$.

Now let $k(2)[-1]\in\ialg_{\Lie}^{gr}$ be the trivial graded dg Lie algebra concentrated in weight $2$ and cohomological degree $1$. Then a bivector $\pi$ satisfying $[\pi, \pi] = 0$ is the same as a morphism of graded Lie algebras $k(2)[-1]\rightarrow \Pol(A, n)[n+1]$. In fact, this gives the whole $\infty$-groupoid of $n$-shfited Poisson structures as shown in \cite{Mel}.

\begin{thm}
One has an equivalence of $\infty$-groupoids
\[\Pois(A, n)\cong \Map_{\ialg_{\Lie}^{gr}}(k(2)[-1], \Pol(A, n)[n+1]).\]
\end{thm}

To get a more concrete understanding of the previous statement, let us recall the notion of a Maurer--Cartan space from \cite{Hin} and \cite{Get}. Let $\Omega_n = \Omega(\Delta^n)$ be the commutative dg algebra of polynomial differential forms on the $n$-simplex. Explicitly, it is given by
\[\Omega_n = k[t_0,\dots, t_n, \d t_0,\dots, \d t_n] / (\sum t_i = 1,\ \sum \d t_i = 0),\]
i.e. the de Rham algebra of $k[t_0,\dots, t_n] / (\sum t_i = 1)$, the commutative dg algebra of polynomial functions on the $n$-simplex. The simplices $\Delta^\bullet$ form a cosimplicial simplicial set, so $\Omega_\bullet$ form a simplicial commutative dg algebra. Therefore, if $\g$ is a dg Lie algebra, $\g\otimes \Omega_\bullet$ is a simplicial dg Lie algebra.

\begin{defn}
Let $\g$ be a nilpotent dg Lie algebra.
\begin{itemize}
\item The \defterm{set of Maurer--Cartan elements of $\g$} is the set $\MC(\g)$ of elements $x\in\g$ of degree $1$ satisfying
\[\d x + \frac{1}{2}[x, x] = 0.\]

\item The \defterm{Maurer--Cartan space of $\g$} is the simplicial set
\[\uMC_\bullet(\g) = \MC(\g\otimes\Omega_\bullet).\]
\end{itemize}
\end{defn}

The term ``Maurer--Cartan space'' indicates that we have an $\infty$-groupoid, i.e. the simplicial set $\uMC(\g)$ is a Kan complex. This is a theorem of Hinich and Getzler, see \cite[Proposition 4.7]{Get}.

\begin{example}
Suppose $\g$ is a nilpotent Lie algebra concentrated in degree $0$. Then $\uMC(\g)$ is weakly equivalent to the nerve of the groupoid $\ast / \exp(\g)$, where $\exp(\g)$ is the unipotent group with Lie algebra $\g$. Thus, $\uMC(-)$ can be thought of as an integration functor.
\end{example}

\begin{defn}
Let $\g$ be a dg Lie algebra. It is \defterm{pronilpotent} if it admits a cofiltration
\[\g_0\twoheadleftarrow\g_1\twoheadleftarrow \dots\]
where $\g = \lim_n \g_n$ and $\g_n$ is nilpotent.
\end{defn}

\begin{defn}
Let $\g$ be a pronilpotent dg Lie algebra. The \defterm{Maurer--Cartan space of $\g$} is the simplicial set
\[\uMC(\g) = \lim_n \uMC(\g_n).\]
\end{defn}

Again, for a pronilpotent dg Lie algebra, $\uMC(\g)$ is a Kan complex by \cite[Proposition 4.1]{DR} and \cite[Proposition 2.6]{Yal}. Given a graded dg Lie algebra $\g$, its completion $\g^{\geq 2}$ in weights at least 2 is pronilpotent: it is obtained as the inverse limit of the system
\[\g^{\geq 2}/ \g^{\geq 3} \twoheadleftarrow \g^{\geq 2}/ \g^{\geq 4} \twoheadleftarrow \dots\]
of nilpotent dg Lie algebras.

Note that a morphism of graded Lie algebras $k(2)[-1]\rightarrow \g$ is the same as an element $x_2\in \g$ of weight $2$ and cohomological degree $1$ which satisfies $[x_2, x_2] = 0$ up to coherent higher homotopies. The following statement proved in \cite[Proposition 1.19]{MS1} shows that the higher homotopies can be encoded by finding $x = x_2 + x_3 + \dots$ where $x_n\in\g$ has weight $n$ such that $x$ satisfies the Maurer--Cartan equation.

\begin{prop}
Let $\g$ be a graded dg Lie algebra. Then we have a natural equivalence of $\infty$-groupoids
\[\Map_{\ialg_{\Lie}^{gr}}(k(2)[-1], \g)\cong \uMC(\g^{\geq 2}).\]
\end{prop}

\subsection{Globalization}
\label{sect:globalization}

So far we have defined the $\infty$-groupoid of $n$-shifted Poisson structures $\Pois(A, n)$ on a commutative dg algebra $A$. We are now going to sketch how this definition extends to the general setting of derived Artin stacks. Such an extension will be provided using the technology of \emph{formal localization} developed in \cite{CPTVV}. We refer the reader to \cite{PV} for a more detailed account of formal localization, while here we will just indicate some general ideas.

Let $\ich^{\leq 0}$ be the $\infty$-category of chain complexes of $k$-vector spaces concentrated in non-positive cohomological degrees (i.e. connective chain complexes). The $\infty$-category of connective commutative dg algebras is denoted by $\icalg(\ich^{\leq 0})$. There is a notion of \'{e}tale maps and covers in $\icalg(\ich^{\leq 0})$.

\begin{defn}$ $
\begin{itemize}
\item A \defterm{derived prestack} is a functor of $\infty$-categories $X\colon \icalg(\ich^{\leq 0})\rightarrow \cS$ where $\cS$ is the $\infty$-category of spaces.

\item A \defterm{derived stack} is a derived prestack $X$ which satisfies \'{e}tale descent: for any \'{e}tale cover $A\rightarrow B$ in $\icalg(\ich^{\leq 0})$ the natural morphism
\[X(A)\longrightarrow \Tot(\cosimp{X(B)}{X(B\otimes_A B)})\]
is an equivalence.
\end{itemize}
\end{defn}

The following definition is an informal rephrasing of \cite[Definition 1.3.1]{HAGII}:

\begin{defn}$ $
\begin{itemize}
\item A derived stack $X$ is \defterm{$(-1)$-geometric} if it is equivalent to an affine derived scheme, i.e. it is a representable functor $\icalg(\ich^{\leq 0})\rightarrow \cS$.

\item A derived stack $X$ is \defterm{$n$-geometric} if the diagonal $X\rightarrow X\times X$ is $(n-1)$-geometric and it admits a smooth cover $\coprod U_i\rightarrow X$ where $U_i$ are affine.

\item A derived stack $X$ is a \defterm{derived Artin stack} (sometimes just called a geometric stack) if it is $n$-geometric for some $n$.
\end{itemize}
\end{defn}

Alternatively, one can say that $X$ is $n$-geometric if there is a smooth groupoid
\[\simp{Y_0}{Y_1}\]
where each $Y_i$ is $(n-1)$-geometric and $X$ is equivalent to $\colim Y_\bullet$, the quotient of $Y_0$ by this groupoid.

Let us also recall the notion of a derived stack locally of finite presentation:

\begin{defn}
A derived stack $X$ is \defterm{locally of finite presentation} if for any filtered system of connective commutative dg algebras $A_i\in\icalg(\ich^{\leq 0})$ the natural morphism
\[\colim X(A_i)\longrightarrow X(\colim A_i)\]
is an equivalence.
\end{defn}

Given a commutative dg algebra $A\in\icalg(\ich^{\leq 0})$ and an $A$-module $M$, denote by $A\oplus M$ the corresponding square-zero extension. Let us recall that the cotangent complex $\bL_A\in\imod_A$ is defined by the universal property
\[\Map_{\imod_A}(\bL_A, M)\cong \Map_{\icalg(\ich^{\leq 0})_{/A}}(A, A\oplus M),\]
where $\icalg(\ich^{\leq 0})_{/A}$ is the $\infty$-category of connective commutative dg algebras with a map to $A$. Note that the right-hand side simply represents derivations from $A$ to $M$.

Similarly, for a derived stack $X$ one can define an analogous universal property for the cotangent complex $\bL_X\in\QCoh(X)$, but this object is not guaranteed to exist.

\begin{prop}
Suppose $X$ is a derived Artin stack. Then it admits a cotangent complex $\bL_X\in\QCoh(X)$. If we assume moreover that $X$ is locally of finite presentation, then $\bL_X$ is perfect.
\end{prop}
\begin{proof}
The first statement is given by \cite[Corollary 2.2.3.3]{HAGII}.

To prove the second statement, consider $f\colon S\rightarrow X$ where $S$ is an affine derived scheme. Consider a filtered diagram $\cM_\bullet\colon I\rightarrow \QCoh(S)$ and let us denote by \[S[\cM_i]=\Spec(\cO(S)\oplus \cM_i)\] the square-zero extension.

Denote by $\Map_{S/}(-, -)$ the mapping space in the $\infty$-category of derived stacks under $S$. By the universal property of the cotangent complex,
\[\Map_{S/}(S[\cM_i], X) \cong \Map_{\QCoh(S)}(f^*\bL_X, \cM_i).\]

Therefore, we get a diagram of spaces
\[
\xymatrix{
\colim \Map_{S/}(S[\cM_i], X) \ar[r] \ar^{\sim}[d] & \Map_{S/}(S[\colim \cM_i], X) \ar^{\sim}[d] \\
\colim \Map_{\QCoh(S)}(f^*\bL_X, \cM_i) \ar[r] & \Map_{\QCoh(S)}(f^*\bL_X, \colim \cM_i)
}
\]

So, if $X$ is locally of finite presentation, the top horizontal morphism is an equivalence and hence the bottom horizontal morphism is an equivalence, i.e. $f^*\bL_X\in\QCoh(S)$ is compact and hence dualizable.

By \cite[Corollary 4.6.1.11]{HA} this implies that $\bL_X\in\QCoh(X)$ itself is dualizable, i.e. it is perfect.
\end{proof}

\begin{example}
Suppose $G$ is an affine algebraic group. Its classifying stack $\B G$ is defined to be the colimit of the groupoid $G\rightrightarrows \pt$ in the $\infty$-category of derived stacks. Note that by construction it is a derived Artin stack. One has an equivalence of symmetric monoidal categories
\[\QCoh(\B G)\cong \Rep G,\]
where $\Rep G$ is the dg category of $G$-representations. Under this equivalence the cotangent complex $\bL_{\B G}\in\QCoh(\B G)$ corresponds to $\g^*[-1]\in\Rep G$, the coadjoint representation in cohomological degree $1$.
\label{ex:BG}
\end{example}

\begin{defn}
Suppose $X$ is a derived stack which admits a perfect cotangent complex. Its \defterm{tangent complex} is
\[\bT_X = \bL_X^\vee = \Hom(\bL_X, \cO_X)\in\QCoh(X).\]
\end{defn}

\begin{defn}
Let $X$ be a derived prestack. Its \defterm{de Rham prestack $X_{\dR}$} is given by
\[X_{\dR}(S) = X(\H^0(S)_{red}).\]
\end{defn}

The construction $X\mapsto X_{\dR}$ preserves limits, so if $X$ is a derived stack, then so is $X_{\dR}$. Note, however, that $X_{\dR}$ is almost never a derived Artin stack. Nevertheless, it admits a cotangent complex which is in fact trivial since by definition $X_{\dR}$ does not have infinitesimal deformations.

\begin{prop}
Let $X$ be a derived prestack. Then $\bL_{X_{\dR}} = 0$.
\label{prop:dRcotangentcomplex}
\end{prop}

\begin{example}
Suppose $X$ is a smooth scheme. In particular, it is formally smooth, i.e. $p\colon X\rightarrow X_{\dR}$ is an epimorphism. Therefore,
\[X_{\dR}\cong \colim \left(\simp{X}{X\times_{X_{\dR}} X}\right)\]

We can identify $X\times_{X_{\dR}} X\cong \widehat{X\times X}_\Delta$, the formal completion of $X\times X$ along the diagonal. Therefore, $X_{\dR}$ is a quotient of $X$ by its infinitesimal groupoid.
\end{example}

As the above example shows, $X_{\dR}$ is a kind of a formal derived stack. Such objects were studied in \cite{GR} and \cite[Section 2.1]{CPTVV}. Consider the projection $p\colon X\rightarrow X_{\dR}$. Even though $p$ is not affine, it is a family of formal affine stacks (see \cite[Definition 2.1.5]{CPTVV}). It was shown in \cite{CPTVV} that one can enhance $p_*\cO_X$ to a prestack of \emph{graded mixed commutative dg algebras} so that many geometric structures on $X$ can be recovered from it. We can summarize it in the following slogan:
\begin{itemize}
\item A formal affine stack is determined by its graded mixed commutative dg algebra of global functions.
\end{itemize}

\begin{defn}
A \defterm{graded mixed complex} is a graded complex $V=\bigoplus_n V(n)$ equipped with a square-zero endomorphism $\epsilon$ of weight $1$ and cohomological degree $1$.
\end{defn}

Denote by $\ich^{gr, \epsilon}$ the $\infty$-category of graded mixed complexes.

\begin{defn}
A \defterm{graded mixed commutative dg algebra} is a commutative algebra in $\ich^{gr, \epsilon}$.
\end{defn}

\begin{defn}
Suppose $A$ is a graded mixed commutative dg algebra.
\begin{itemize}
\item Its \defterm{realization} is
\[|A| = \prod_{n\geq 0} A(n)\]
equipped with the differential $\d_A+\epsilon$.

\item Its \defterm{Tate realization} is
\[|A|^t = \underset{m\rightarrow-\infty}\colim\prod_{n\geq m} A(n)\]
equipped with the differential $\d_A+\epsilon$.
\end{itemize}
\end{defn}

One can also encode Tate realization in terms of the usual realization as follows. There is a certain ind-object in graded mixed complexes $k(\infty)$ which is defined in \cite[Section 1.5]{CPTVV}. Given any other graded mixed complex $V$, we denote
\[V(\infty) = V\otimes k(\infty).\]
Then we can identify
\[|A|^t\cong |A(\infty)|.\]
Let us stress that the introduction of the twist by $k(\infty)$ will merely allow us to omit Tate realizations in what follows.

Given a formal affine stack $X$, \cite{CPTVV} introduce a graded mixed commutative dg algebra $\bD(X)$ such that
\[|\bD(X)|\cong \cO(X).\]

\begin{defn}
Let $X$ be a derived stack.

\begin{itemize}
\item The \defterm{crystalline structure sheaf} $\bD_{X_{\dR}}$ is a prestack of graded mixed commutative dg algebras on $X_{\dR}$ defined as
\[(\Spec A\rightarrow X_{\dR})\mapsto \bD(\Spec A).\]

\item The \defterm{sheaf of principal parts} $\cB_X$ is a prestack of graded mixed commutative dg algebras on $X_{\dR}$ defined as
\[(\Spec A\rightarrow X_{\dR})\mapsto \bD(X\times_{X_{\dR}} \Spec A).\]
\end{itemize}
\end{defn}

Note that by construction we have a morphism $\bD_{X_{\dR}}\rightarrow \cB_X$. Moreover, we have equivalences of prestacks of commutative dg algebras on $X_{\dR}$
\[|\cB_X|\cong p_*\cO_X,\qquad |\bD_{X_{\dR}}|\cong \cO_{X_{\dR}}.\]

If $X$ is a derived Artin stack locally of finite presentation, many geometric structures on $X$ can be reconstructed from $\bD_{X_{\dR}}\rightarrow \cB_X$: the dg category of perfect complexes $\Perf(X)$, cotangent complex $\bL_X$ etc. We refer to \cite{PV} and \cite[Section 2]{CPTVV} for details.

We are now ready to define shifted Poisson structures on a derived Artin stack.

\begin{defn}
Let $X$ be a derived Artin stack locally of finite presentation. The \defterm{$\infty$-groupoid $\Pois(X, n)$ of $n$-shifted Poisson structures on $X$} is
\[\Pois(X, n) = \Pois(\cB_X(\infty)/\bD_{X_{\dR}}(\infty), n).\]
\label{defn:nshiftedPoisson}
\end{defn}

In other words, an $n$-shifted Poisson structure on $X$ is defined to be a lift of $\cB_X(\infty)$ from a $\bD_{X_{\dR}}(\infty)$-linear commutative algebra to a $\bD_{X_{\dR}}(\infty)$-linear $\bP_{n+1}$-algebra.

\begin{remark}
In the definition of $\Pois(X, n)$ we consider $\cB_X(\infty)$ as a commutative algebra in the dg category of prestacks of graded mixed $\bD_{X_{\dR}}(\infty)$-modules on $X_{\dR}$. See also \cref{remark:generalmodelcategory}.
\end{remark}

\begin{remark}
One may also consider a ``non-Tate'', version of $n$-shifted Poisson structures on $X$ defined as $\Pois(\cB_X / \bD_{X_{\dR}}, n)$. Even though the two definitions agree for derived schemes, they are different in general and, for instance, \cref{prop:poissonpolyvectors} fails for $X=\B G$ as the non-Tate version does not ``see'' all polyvectors.
\end{remark}

In a way, \cref{defn:nshiftedPoisson} asserts that an $n$-shfited Poisson structure on $X$ is the same as a fiberwise $n$-shifted Poisson structure on $X\rightarrow X_{\dR}$. Such a definition is sensible since the natural inclusion of fiberwise $n$-shifted polyvector fields along $X\rightarrow X_{\dR}$ into all $n$-shifted polyvector fields on $X$ is an equivalence since $\bT_{X_{\dR}}=0$ by \cref{prop:dRcotangentcomplex}.

Similarly, under the same assumption one can define a graded $\bP_{n+2}$-algebra $\Pol(X, n)$ of $n$-shifted polyvectors on $X$ using $\cB_X$.

\begin{prop}
One has an equivalence of graded commutative dg algebras
\[\Pol(X, n)\cong \Gamma(X, \Sym(\bT_X[-n-1])).\]
\label{prop:polyvectorsgrcdga}
\end{prop}

\begin{prop}
One has an equivalence of spaces
\[\Pois(X, n)\cong \Map_{\ialg_{\Lie}^{gr}}(k(2)[-1], \Pol(X, n)[n+1]).\]
\label{prop:poissonpolyvectors}
\end{prop}

\begin{example}
Suppose $X$ is a (possibly singular) scheme over $k$. Its cotangent complex $\bL_X$ is connective, so the weight $p$ part of $\Pol(X, n)$ is concentrated in cohomological degrees $[p(n+1), \infty)$. Therefore, by \cref{prop:poissonpolyvectors} the spaces $\Pois(X, n)$ are contractible for $n>0$, i.e. every $n$-shifted Poisson structure for $n>0$ is canonically zero. The same computation shows that $\Pois(X, 0)$ is identified with the \emph{set} of ordinary Poisson structures on $X$.
\label{ex:poissonscheme}
\end{example}

\begin{remark}
The book \cite{GR} presents a slightly different formalism for treating formal stacks in terms of inf-schemes. Given a derived Artin stack $X$, the morphism $p\colon X\rightarrow X_{\dR}$ is proper, so $p_*$ is left adjoint to the symmetric monoidal functor
\[p^!\colon \IndCoh(X_{\dR})\rightarrow \IndCoh(X),\]
where $\IndCoh(-)$ is the $\infty$-category of \emph{ind-coherent sheaves}. In particular, $p_*$ acquires an oplax symmetric monoidal structure. Therefore, we get a cocommutative coalgebra
\[p_*\omega_X\in\IndCoh(X_{\dR}).\]
One may define an $n$-shifted Poisson structure on $X$ as a lift of the cocommutative structure on $p_*\omega_X$ to a $\bP_{n+1}$-coalgebra structure. It would be interesting to compare such a definition to \cref{defn:nshiftedPoisson}. Such a definition might be useful to deal with derived stacks not locally of finite presentation.
\end{remark}

Recall that in \cref{sect:additivity} we have defined a forgetful functor
\[\forget_{n+1}^{n-m+1}\colon \ialg_{\bP_{n+1}}\rightarrow \ialg_{\bP_{n-m+1}}.\]
Since it is compatible with the forgetful functors to $\icalg$, it induces a forgetful map
\begin{equation}
\Pois(X, n)\longrightarrow \Pois(X, n-m)
\label{eq:forgetshiftstack}
\end{equation}
for any derived stack $X$. By \cref{prop:poissonforget2}, the forgetful map \eqref{eq:forgetshiftstack} is trivial for $m\geq 2$, i.e. it sends any $n$-shifted Poisson structure on $X$ to the zero $(n-m)$-shifted Poisson structure on $X$. We refer to \cref{conj:atiyahbracket} and \cref{prop:forgetshiftBG} for examples of the forgetful map $\Pois(X, n)\rightarrow \Pois(X, n-1)$.

\section{Coisotropic and Lagrangian structures}

In the second lecture we explain how to define $n$-shifted coisotropic structures in terms of relative Poisson algebras and, after a brief reminder on shifted sympletic geometry, explain the relationship between shifted Poisson and shifted symplectic structures.

\subsection{Relative Poisson algebras}
\label{sect:relpoisson}

Recall the equivalence of $\infty$-categories from \cref{thm:poissonadditivity}:
\[\ialg_{\bP_{n+1}}\cong \ialg(\ialg_{\bP_n}).\]

Given a symmetric monoidal $\infty$-category $\cC$ we denote by $\ilmod(\cC)$ the $\infty$-category of pairs $(A, M)$ of an associative algebra $A$ in $\cC$ and an $A$-module $M$. By definition (see \cite[Definition 5.3.1.6]{HA}), if $\cC$ admits centralizers, we can identify $\ilmod(\cC)^\sim$ with the space of pairs of an associative algebra $A\in\ialg(\cC)$, an object $M\in\cC$ and a morphism of associative algebras $A\rightarrow \Z(M)$. We will now give a similar characterization of the whole $\infty$-category $\ilmod(\ialg_{\bP_n})$.

\begin{defn}
Let $B\in\alg_{\bP_n}$ be a $\bP_n$-algebra. Its \defterm{strict center} is
\[\Z^{str}(B) = \Hom_{\mod_B}(\Sym_B(\Omega^1_B[n]), B)\]
with the differential twisted by $[\pi_B, -]$.
\end{defn}

The strict center is equipped with a $\bP_{n+1}$-algebra structure where the bracket is a generalization of the Schouten bracket of polyvector fields. The term ``strict center'' is explained by the following proposition which will be proved in a future paper.

\begin{prop}
Assume $B\in\alg_{\bP_n}$ is cofibrant as a commutative dg algebra. Then we have an equivalence of objects of $\ialg_{\bP_{n+1}}\cong \ialg(\ialg_{\bP_n})$:
\[\Z^{str}(B)\cong \Z(B),\]
where $\Z(B)$ is the center of $B$ in the sense of \cref{defn:center}.
\end{prop}

We are now ready to define relative Poisson algebras introduced in \cite{Sa1} under the name ``coisotropic morphism''.

\begin{defn}
A \defterm{$\bP_{[n+1, n]}$-algebra} is a triple $(A, B, F)$, where
\begin{itemize}
\item $A$ is a $\bP_{n+1}$-algebra,

\item $B$ is a $\bP_n$-algebra,

\item $F\colon A\rightarrow \Z^{str}(B)$ is a morphism of $\bP_{n+1}$-algebras.
\end{itemize}
\end{defn}

The above definition gives rise to a two-colored dg operad $\bP_{[n+1, n]}$ of relative Poisson algebras. The following statement is proved in \cite[Theorem 3.8]{Sa2}.

\begin{prop}
One has an equivalence of $\infty$-categories
\[\ialg_{\bP_{[n+1, n]}}\cong \ilmod(\ialg_{\bP_n}).\]
\end{prop}

One can introduce a forgetful functor (see \cite[Section 3.5]{MS1})
\begin{equation}
\U\colon \ialg_{\bP_{[n+1, n]}}\longrightarrow \ialg_{\bP^{nu}_{n+1}}
\label{eq:relpoissonforget}
\end{equation}
from relative Poisson algebras to non-unital $\bP_{n+1}$-algebras. Given $(A, B)\in\ialg_{\bP_{[n+1, n]}}$, this construction has the following properties:
\begin{enumerate}
\item One can identify as non-unital commutative dg algebras
\[\U(A, B)\cong \fib(A\rightarrow B).\]

\item One has a fiber sequence of Lie algebras
\[B[n-1]\longrightarrow \U(A, B)[n]\longrightarrow A[n],\]
i.e. we have a pullback square
\[
\xymatrix{
B[n-1] \ar[r] \ar[d] & \U(A, B) \ar[d] \\
0 \ar[r] & A[n]
}
\]
\end{enumerate}

The natural projection $\Z(B)\rightarrow B$ is a morphism of commutative dg algebras, so we also get a forgetful functor
\[\alg_{\bP_{[n+1, n]}}\longrightarrow \Arr(\calg),\]
where $\Arr(\cC) = \Fun(\Delta^1, \cC)$. After localization it induces a functor of $\infty$-categories
\[\ialg_{\bP_{[n+1, n]}}\longrightarrow \Arr(\icalg).\]

\begin{defn}
Let $f\colon A\rightarrow B$ be a morphism of commutative dg algebras. The \defterm{space $\Cois(f, n)$ of $n$-shifted coisotropic structures on $f$} is the fiber of
\[\ialg_{\bP_{[n+1, n]}}^{\sim}\longrightarrow \Arr(\icalg)^{\sim}\]
at $f\in\Arr(\icalg)$.
\end{defn}

By construction we have forgetful maps
\[
\xymatrix{
& \Cois(f, n) \ar[dl] \ar[dr] & \\
\Pois(B, n-1) && \Pois(A, n)
}
\]

We have a natural involution
\[\opp\colon \Pois(B, n-1)\longrightarrow \Pois(B, n-1)\]
induced by the automorphism of the operad $\bP_n$ given by flipping the sign of the Poisson bracket. We denote by $\overline{B}$ the opposite $\bP_n$-algebra.

The next statement shows that a derived intersection of $n$-shifted coisotropic morphisms carries an $(n-1)$-shifted Poisson structure.

\begin{prop}
Suppose $A\rightarrow B_1$ and $A\rightarrow B_2$ are two morphisms of commutative dg algebras. We have a morphism of spaces
\[\Cois(A\rightarrow B_1, n)\times_{\Pois(A, n)} \Cois(A\rightarrow B_2, n) \longrightarrow \Pois(B_1\otimes_A B_2, n-1).\]

Moreover, given two $n$-shifted coisotropic morphisms $A\rightarrow B_1$ and $A\rightarrow B_2$, the projection $\overline{B_1}\otimes B_2\rightarrow B_1\otimes_A B_2$ is compatible with the $(n-1)$-shifted Poisson structures on both sides.
\label{prop:affinecoisotropicintersection}
\end{prop}
\begin{proof}
By \cref{prop:oppositepoisson}, we have a diagram of $\infty$-categories
\[
\xymatrix{
\ilmod(\ialg_{\bP_n}) \ar^{\sim}[r] \ar[d] & \irmod(\ialg_{\bP_n}) \ar[d] \\
\ialg_{\bP_n} \ar^{\opp}[r] & \ialg_{\bP_n}
}
\]

The relative tensor product \cite[Section 4.4]{HA} gives rise to horizontal functors
\[
\xymatrix{
\irmod(\ialg_{\bP_n})\times_{\ialg(\ialg_{\bP_n})} \ilmod(\ialg_{\bP_n}) \ar^-{\sim}[r] \ar[d] & \ibimod(\ialg_{\bP_n}) \ar[r] \ar[d] & \ialg_{\bP_n} \ar[d] \\
\irmod(\icalg)\times_{\ialg(\icalg)} \ilmod(\icalg) \ar^-{\sim}[r] & \ibimod(\icalg) \ar[r] & \icalg
}
\]

Using \cite[Proposition 2.4.3.9]{HA} we can identify
\[\ilmod(\icalg)\cong \Arr(\icalg),\qquad \irmod(\icalg)\cong \Arr(\icalg)\]
and the claim follows by passing to fibers of the vertical functors.
\end{proof}

Given a morphism $f\colon L\rightarrow X$ of derived Artin stacks locally of finite presentation we can define the space $\Cois(f, n)$ using the formalism of formal localization as in \cref{sect:globalization}. Moreover, the previous statement admits a generalization to derived Artin stacks, see \cite[Theorem 3.6]{MS2}.

\begin{prop}
Suppose $L_1, L_2\rightarrow X$ are two morphisms of derived Artin stacks locally of finite presentation. Then we have a diagram of spaces
\[
\xymatrix{
\Cois(L_1\rightarrow X, n)\times_{\Pois(X, n)} \Cois(L_2\rightarrow X, n) \ar[r] \ar[d] & \Pois(L_1\times_X L_2, n-1) \ar[d] \\
\Pois(L_1, n-1)\times \Pois(L_2, n-1) \ar^{\opp\times\id}[r] & \Pois(L_1, n-1)\times \Pois(L_2, n-1)
}
\]
\label{prop:coisotropicintersection}
\end{prop}

This is known as a coisotropic intersection theorem.

\subsection{Poisson morphisms}

\begin{defn}
Let $f\colon A_1\rightarrow A_2$ be a morphism of commutative dg algebras. The \defterm{space $\Pois(f, n)$ of $n$-shifted Poisson structures on $f$} is the fiber of
\[\Arr(\ialg_{\bP_{n+1}})^{\sim} \longrightarrow \Arr(\icalg)^{\sim}\]
at $f\in\Arr(\icalg)$.
\end{defn}

One can similarly define the space $\Pois(f, n)$ for a morphism of derived Artin stacks locally of finite presentation.

Recall that a morphism of smooth Poisson schemes $f\colon X_1\rightarrow X_2$ is Poisson iff its graph $g\colon X_1\rightarrow \overline{X_1}\times X_2$ is coisotropic, where $\overline{X_1}$ is $X_1$ equipped with the opposite Poisson structure. We have a similar statement in the derived context as well, see \cite[Theorem 2.8]{MS2}.

\begin{prop}
Let $f\colon X_1\rightarrow X_2$ be a morphism of derived Artin stacks locally of finite presentation and let $g\colon X_1\rightarrow X_1\times X_2$ be its graph. One has a fiber square of spaces
\[
\xymatrix{
\Pois(f, n) \ar[r] \ar[d] & \Pois(X_1, n)\times \Pois(X_2, n) \ar^{\id\oplus\opp}[d] \\
\Cois(g, n) \ar[r] & \Pois(X_1\times X_2, n)
}
\]
\end{prop}

\subsection{Relative polyvectors}

In this section we generalize results of \cref{sect:polyvectors} to the relative setting.

Let $f\colon A\rightarrow B$ be a morphism of commutative dg algebras. The center $\Z(\triv_{\bP_n}B)$ is a graded $\bP_{n+1}$-algebra and we get a natural morphism of graded $\bP_{n+1}$-algebras
\[\triv_{\bP_{n+1}}(A)\longrightarrow \Z(\triv_{\bP_n}(B)).\]

Also recall that we denote by $\Pol(B/A, n-1)$ the graded $\bP_{n+1}$-algebra of polyvectors of $B\in\icalg(\imod_A)$.

\begin{prop}
One has an equivalence of graded $\bP_{n+1}$-algebras
\[\Pol(B/A, n-1)\cong \Z(\triv_{\bP_{n+1}}(A)\rightarrow \Z(\triv_{\bP_n}(B))).\]
\end{prop}

From the previous statement and the universal property of centralizers we obtain that
\[(\Pol(A, n), \Pol(B/A, n-1))\in \ilmod(\ialg_{\bP_{n+1}}^{gr}).\]

\begin{defn}
Let $f\colon A\rightarrow B$ be a morphism of commutative dg algebras. The \defterm{algebra of relative $n$-shifted polyvector fields} is
\[\Pol(f, n) = \U(\Pol(A, n), \Pol(B/A, n-1))\in\ialg_{\bP^{nu}_{n+1}}^{gr}.\]
\end{defn}

Given this definition, we have the following explicit way to compute the space of $n$-shifted coisotropic structures.

\begin{prop}
Let $f\colon A\rightarrow B$ be a morphism of commutative dg algebras. One has an equivalence of spaces
\[\Cois(f, n)\cong \Map_{\ialg_{\Lie}^{gr}}(k(2)[-1], \Pol(f, n)[n+1]).\]
\end{prop}

As in \cref{sect:globalization}, we can generalize relative polyvectors to morphisms of derived Artin stacks.

\begin{prop}
Let $f\colon L\rightarrow X$ be a morphism of derived Artin stacks locally of finite presentation. The image of the pair $(\Pol(X, n), \Pol(L/X, n-1))$ under the functor $\ialg^{gr}_{\bP_{[n+2, n+1]}}\rightarrow \Arr(\icalg^{gr})$ is equivalent to the composite
\[\Gamma(X, \Sym(\bT_X[-n-1]))\longrightarrow \Gamma(L, \Sym(f^*\bT_X[-n-1]))\longrightarrow \Gamma(L, \Sym(\bT_{L/X}[-n])).\]
\end{prop}

\begin{prop}
Let $f\colon L\rightarrow X$ be as before. Then one has an equivalence of spaces
\[\Cois(f, n)\cong \Map_{\ialg_{\Lie}^{gr}}(k(2)[-1], \Pol(f, n)[n+1]).\]
\label{prop:coisotropicpolyvectors}
\end{prop}

\begin{example}
Let $f\colon L\rightarrow X$ be an embedding of a closed subscheme. As in \cref{ex:poissonscheme}, from  we see that the weight $p$ part of $\Pol(f, n)$ is concentrated in cohomological degrees $[p(n+1), \infty)$. Therefore, by \cref{prop:coisotropicpolyvectors} the spaces $\Cois(f, n)$ are contractible for $n>0$. Moreover, for $n=0$ we recover the \emph{set} of Poisson structures on $X$ for which $L$ is coisotropic in the usual sense, i.e. for which the bivector $\pi_X\in\Gamma(X, \wedge^2 \T_X)$ vanishes under the restriction
\[\Gamma(X, \wedge^2 \T_X)\longrightarrow \Gamma(L, \wedge^2 \N_{L/X}).\]
\label{ex:coisotropicscheme}
\end{example}

\begin{example}
If $X$ is a smooth Poisson scheme, the identity inclusion $X\rightarrow X$ is coisotropic. In the derived setting one can show that there is a \emph{unique} coisotropic structure.

Let $X$ be an $n$-shifted Poisson stack and consider the identity map $\id\colon X\rightarrow X$. By the property (2) of the forgetful functor $\U\colon \ialg_{\bP_{[n+2, n+1]}}\rightarrow \ialg_{\bP_{n+2}^{nu}}$ given in \cref{sect:relpoisson} we have a fiber sequence of Lie algebras
\[\Pol(X/X, n-1)[n]\longrightarrow \Pol(\id, n)[n+1]\longrightarrow \Pol(X, n)[n+1].\]

By \cref{prop:polyvectorsgrcdga} we can identify $\Pol(X/X, n-1)\cong \Gamma(X, \Sym(\bT_{X/X}[-n]))$ as graded commutative dg algebras and since $\bT_{X/X}=0$, the graded complex $\Pol(X/X, n-1)$ is concentrated in weight $0$. In particular, the morphism $\Pol(\id, n)^{\geq 2}\rightarrow \Pol(X, n)^{\geq 2}$ is an equivalence and hence the forgetful map
\[\Cois(\id, n)\longrightarrow \Pois(X, n)\]
is an equivalence of spaces. In other words, there is a unique $n$-shifted coisotropic structure on the identity map compatible with the given $n$-shifted Poisson structure on $X$.
\end{example}

Using the previous example we get a diagram $\Pois(X, n)\xleftarrow{\sim} \Cois(\id, n)\rightarrow \Pois(X, n-1)$ and hence a forgetful map $\Pois(X, n)\rightarrow \Pois(X, n-1)$.

\begin{prop}
The forgetful map $\Pois(X, n)\rightarrow \Pois(X, n-1)$ defined above is equivalent to the one given by \eqref{eq:forgetshiftstack}.
\end{prop}

\subsection{Shifted symplectic structures}

In this section we briefly remind some facts about shifted symplectic structures. The reader is referred to \cite{PTVV} and \cite{Cal2} for a more complete treatment.

\begin{defn}
Let $A$ be a commutative dg algebra. Its \defterm{de Rham complex} is the graded mixed commutative dg algebra
\[\DR(A) = \Sym_A(\bL_A[-1])\]
with weights given by the $\Sym$ grading and the mixed structure given by the de Rham differential.
\end{defn}

\begin{defn}
Let $X$ be a derived stack. Its \defterm{de Rham complex} is the graded mixed commutative dg algebra
\[\DR(X) = \lim_{\Spec A\rightarrow X} \DR(A).\]
\end{defn}

\begin{defn}
Let $X$ be a derived stack.
\begin{itemize}
\item The \defterm{space of two-forms of degree $n$ on $X$} is
\[\cA^2(X, n) = \Map_{\ich^{gr}}(k(2)[-1], \DR(X)[n+1]).\]

\item The \defterm{space of closed two-forms of degree $n$ on $X$} is
\[\cA^{2, cl}(X, n) = \Map_{\ich^{gr,\epsilon}}(k(2)[-1], \DR(X)[n+1]).\]
\end{itemize}
\end{defn}

From the definition we get a morphism
\[\cA^{2, cl}(X, n)\longrightarrow \cA^2(X, n)\]
which extracts the underlying two-form.

Note that by pulling back differential forms to affines we obtain a natural morphism of \emph{graded commutative dg algebras}
\begin{equation}
\Gamma(X, \Sym(\bL_X[-1]))\longrightarrow \DR(X).
\label{eq:DRaffinepullback}
\end{equation}
Unfortunately, for a general derived stack $X$ we do not know if this morphism is an equivalence. However, we have the following statement, see \cite[Proposition 1.14]{PTVV}.

\begin{prop}
Suppose $X$ is a derived Artin stack. Then the morphism of graded commutative dg algebras \eqref{eq:DRaffinepullback} is an equivalence.
\end{prop}

In particular, under the previous assumptions we see that a two-form $\omega\in\cA^2(X, n)$ is the same as a morphism of complexes
\[k\rightarrow \Gamma(X, \Sym^2(\bL_X[-1]))[2+n].\]
If we assume $\bL_X$ is perfect, by adjunction this morphism gives rise to a morphism
\[\omega^\sharp\colon \bT_X\rightarrow \bL_X[n].\]

\begin{defn}
Let $X$ be a derived Artin stack locally of finite presentation. A two-form $\omega\in\cA^2(X, n)$ is \defterm{nondegenerate} if the morphism $\omega^\sharp\colon \bT_X\rightarrow \bL_X[n]$ is an equivalence.
\end{defn}

\begin{defn}
Let $X$ be a derived Artin stack locally of finite presentation. The \defterm{space of $n$-shfited symplectic structures} on $X$ is the subspace $\Symp(X, n)\subset \cA^{2, cl}(X, n)$ of closed two-forms whose underlying two-form is nondegenerate.
\end{defn}

\begin{example}
Suppose $Y$ is a derived Artin stack locally of finite presentation. Then for any $n$ the $n$-shifted cotangent stack
\[\T^*[n]Y\cong \Spec_Y\Sym(\bT_Y[-n])\]
has an exact $n$-shifted symplectic structure by \cite{Cal3}.
\end{example}

If $f\colon L\rightarrow X$ is a morphism of derived stacks, we get an induced morphism of graded mixed commutative dg algebras
\[f^*\colon \DR(X)\longrightarrow \DR(L)\]
given by pulling back differential forms. An $n$-shifted isotropic structure on $f$ is then given by a closed two-form $\omega$ on $X$ of degree $n$ and the nullhomotopy of $f^*\omega$.

\begin{defn}
Let $f\colon L\rightarrow X$ be a morphism of derived stacks. The \defterm{space of $n$-shifted isotropic structures on $f$} is
\[\Isot(X, n) = \Map_{\ich^{gr,\epsilon}}(k(2)[-1], \fib(f^*\colon \DR(X)\rightarrow \DR(L))[n+1]).\]
\end{defn}

Suppose $f\colon L\rightarrow X$ is a morphism of derived Artin stacks locally of finite presentation. It is easy to see that given an $n$-shifted isotropic structure on $f$ we obtain a null-homotopy of the composite
\[\bT_L\longrightarrow f^*\bT_X\xrightarrow{\omega^\sharp} f^*\bL_X[n]\longrightarrow \bL_L[n].\]

\begin{defn}
Let $f\colon L\rightarrow X$ be a morphism of derived Artin stacks locally of finite presentation. The \defterm{space of $n$-shifted Lagrangian structures on $f$} is the subspace $\Lagr(f, n)\subset \Isot(X, n)$ of those $n$-shifted isotropic structures whose underlying closed two-form on $X$ is nondegenerate and such that the sequence in $\QCoh(L)$
\[\bT_L\longrightarrow f^*\bT_X\longrightarrow \bL_L[n]\]
is a fiber sequence.
\end{defn}

Equivalently, an $n$-shifted isotropic structure $\lambda$ induces a morphism $\lambda^\sharp\colon \bT_{L/X}\rightarrow\bL_L[n-1]$ and it is Lagrangian if $\lambda^\sharp$ is an equivalence.

\begin{example}
Suppose $Z\rightarrow Y$ is a morphism of derived Artin stacks locally of finite presentation. Then we have an $n$-shifted conormal stack
\[\N^*[n](Z/Y) = \Spec_Z\Sym (\bT_{Z/Y}[1-n]).\]

We have a natural morphism
\[\Spec_Z \Sym(\bT_{Z/Y}[1-n])\rightarrow \Spec_Z\Sym (f^*\bT_Y[-n])\rightarrow \Spec_Y \Sym (\bT_Y[-n]) = \T^*[n]Y\]
and it is shown in \cite{Cal3} that $\N^*[n](Z/Y)\rightarrow \T^*[n](Y)$ carries an exact $n$-shifted Lagrangian structure.
\label{ex:conormal}
\end{example}

Note that by construction we have a forgetful map
\[\Lagr(f, n)\longrightarrow \Symp(X, n).\]

Fix two maps $f_i\colon L_i\rightarrow X$ for $i=1, 2$. Suppose $X$ carries a closed two-form $\omega$ of degree $n$ and we have two null-homotopies $f_i^*\omega\sim 0$. Then the intersection $L_1\times_X L_2$ carries two null-homotopies of the pullback $g^*\omega$ where $g\colon L_1\times_X L_2\rightarrow X$. The difference between the two null-homotopies is a closed two-form of degree $(n-1)$ on $L_1\times_X L_2$. This procedure in fact preserves nondegeneracy as shown in \cite[Section 2.2]{PTVV}, i.e. given two $n$-shifted Lagrangians $L_1, L_2\rightarrow X$, their intersection $L_1\times_X L_2$ has an $(n-1)$-shifted symplectic structure.

\begin{prop}
Suppose $L_1, L_2\rightarrow X$ are two morphisms of derived Artin stacks locally of finite presentation. Then we have a map of spaces
\[\Isot(L_1\rightarrow X, n)\times_{\cA^{2, cl}(X, n)} \Isot(L_2\rightarrow X, n)\longrightarrow \cA^{2, cl}(L_1\times_X L_2, n-1).\]

This map preserves nondegeneracy, i.e. it induces a map of spaces
\[\Lagr(L_1\rightarrow X, n)\times_{\Symp(X, n)} \Lagr(L_2\rightarrow X, n)\longrightarrow \Symp(L_1\times_X L_2, n-1).\]
\label{prop:lagrangianintersection}
\end{prop}

This is known as the Lagrangian intersection theorem 

\begin{example}
Suppose $X$ is a smooth symplectic scheme and $L\subset X$ is a smooth Lagrangian subscheme. Using the Koszul complex we can identify
\[L\times_X L\cong \N[-1](L/X).\]

Since $L$ is Lagrangian, $\N(L/X)\cong \T^*L$ and hence
\[L\times_X L\cong \T^*[-1] L.\]

The derived scheme $L\times_X L$ is a Lagrangian intersection and so by \cref{prop:lagrangianintersection} it inherits a $(-1)$-shifted symplectic structure. One can show that the equivalence $L\times_X L\cong \T^*[-1] L$ is compatible with the $(-1)$-shifted symplectic structures on both sides.
\end{example}

\subsection{Shifted Poisson and shifted symplectic structures}

If $X$ is a smooth scheme, a symplectic structure on $X$ is the same as a nondegenerate Poisson structure, i.e. a Poisson structure which induces an isomorphism $\T^*_X\rightarrow \T_X$. This is clear on the level of linear algebra and the only nontrivial computation is that closedness of the two-form $\omega$ is equivalent to the Jacobi identity for the bivector $\pi=\omega^{-1}$. In the derived context the same result is far from obvious since we need to show that closedness of the two-form \emph{up to coherent homotopy} is equivalent to the Jacobi identity for the bivector which again holds \emph{up to coherent homotopies}. Nevertheless, we still have an equivalence between nondegenerate $n$-shifted Poisson structures and $n$-shifted symplectic structures.

\begin{defn}
Let $X$ be a derived Artin stack locally of finite presentation.

\begin{itemize}
\item Suppose $\pi$ is an $n$-shifted Poisson structures on $X$. It is \defterm{nondegenerate} if the morphism
\[\pi^\sharp\colon \bL_X\longrightarrow \bT_X[-n]\]
induced by the underlying bivector is an equivalence.

\item Denote by $\Pois^{nd}(X, n)\subset \Pois(X, n)$ the subspace of nondegenerate $n$-shifted Poisson structures.
\end{itemize}
\end{defn}

The following is proved in \cite[Theorem 3.2.4]{CPTVV} and \cite[Theorem 3.13]{Pri1}.

\begin{thm}
One has an equivalence of spaces
\[\Pois^{nd}(X, n)\cong \Symp(X, n).\]
\label{thm:nondegeneratepoisson}
\end{thm}

The proof given in \cite{CPTVV} first proceeds by a reduction to a local statement and then using a strictification provided by a Darboux lemma one shows the equivalence. We will sketch an alternative proof given in \cite{Pri1} in the affine case.

The theorem in both papers is proved by going through an intermediate space of compatible pairs. Let $\pi$ be an $n$-shifted Poisson structure on $A$, a commutative dg algebra. Then $[\pi, -]$ defines a mixed structure on the graded $\bP_{n+2}$-algebra $\Pol(A, n)$. Denote by $\Pol_\pi(A, n)$ the underlying graded mixed commutative dg algebra. By the universal property of the de Rham complex, one has a unique morphism of graded mixed commutative dg algebras
\[\DR(A)\longrightarrow \Pol_\pi(A, n)\]
given by the identity $A\rightarrow A$ in weight $0$ and which sends $f\ddr g\mapsto f[\pi, g]$ in weight $1$.

Splitting the $n$-shifted Poisson structure $\pi=\pi_2+\pi_3+\dots$ according to the weight, we have a natural morphism of graded mixed complexes
\[\sigma\colon k(2)[-1]\longrightarrow \Pol_\pi(A, n)[n+1]\]
given by
\[1\mapsto \sum_{m\geq 2} (m-1) \pi_m.\]

\begin{defn}
Let $\omega\in\cA^{2, cl}(A, n)$ be a closed two-form of degree $n$ on $A$ defining a morphism
\[\omega\colon k(2)[-1]\longrightarrow \DR(X)[n+1]\]
of graded mixed complexes. Let $\pi$ be an $n$-shifted Poisson structure on $A$. The \defterm{compatibility between $\omega$ and $\pi$} is a commutativity data in the diagram of graded mixed complexes
\[
\xymatrix{
k(2)[-1] \ar^{\sigma}[rr] \ar_{\omega}[dr] && \Pol_\pi(X, n) \\
& \DR(X) \ar[ur] &
}
\]
\end{defn}

We denote by $\Comp(A, n)$ the space of such compatible pairs. We have natural forgetful maps
\[
\xymatrix{
& \Comp(A, n) \ar[dl] \ar[dr] & \\
\cA^{2, cl}(A, n) && \Pois(A, n)
}
\]

\begin{lm}
Let $(\omega, \pi)\in\Comp(A, n)$ be a compatible pair on $A$. Then the composite
\[\bL_X\xrightarrow{\pi^\sharp} \bT_X[-n]\xrightarrow{\omega^\sharp} \bL_X\xrightarrow{\pi^\sharp} \bT_X[-n]\]
is homotopic to $\pi^\sharp$.
\label{lm:piomegapi}
\end{lm}

Let $\Comp^{nd}(A, n)\subset \Comp(A, n)$ be the subspace of compatible pairs $(\omega, \pi)$ where $\pi$ is nondegenerate. In particular, by the previous Lemma we get that $\omega$ is nondegenerate as well, so we obtain projections
\[
\xymatrix{
& \Comp^{nd}(A, n) \ar[dl] \ar[dr] & \\
\Symp(A, n) && \Pois^{nd}(A, n).
}
\]

\begin{lm}
The projection $\Comp^{nd}(A, n)\rightarrow \Pois^{nd}(A, n)$ is an equivalence.
\end{lm}
\begin{proof}
If $\pi$ is nondegenerate, the induced morphism of graded mixed commutative dg algebras
\[\DR(A)\longrightarrow \Pol_\pi(A, n)\]
is an equivalence and hence there is a unique extension of $\pi$ to a compatible pair.
\end{proof}

Therefore, we just need to show that $\Comp^{nd}(A, n)\rightarrow \Symp(A, n)$ is an equivalence. The strategy used in \cite{Pri1} is to employ the natural filtrations on both spaces and prove the statement by obstruction theory.

That is, let $\DR^{\geq (m+1)}(A)$ be the truncation of $\DR(A)$ in weights at least $m+1$. Let
\[\Symp^{\leq m}(A, n)\subset \Map_{\ich^{gr, \epsilon}}(k(2)[-1], \DR(A) / \DR^{\geq (m+1)}(A))\]
be the subspace of nondegenerate forms. Since
\[\DR(A) = \lim_m \left(\DR(A) / \DR^{\geq (m+1)}(A)\right),\]
we have
\[\Symp(A, n) = \lim_m \Symp^{\leq m}(A, n).\]

One can similarly define $\Pois^{nd, \leq m}(A, n)$ and $\Comp^{nd, \leq m}(A, n)$. Then we can prove a stronger statement:
\begin{lm}
The projection $\Comp^{nd, \leq m}(A, n)\rightarrow \Symp^{\leq m}(A, n)$ is an equivalence.
\end{lm}

This statement is proved by induction as follows. First, for $m=2$ the statement reduces to a problem in linear algebra: given a nondegenerate two-form $\omega$, by \cref{lm:piomegapi} we have to show there is a unique non-degenerate bivector $\pi$ such that
\[\pi^\sharp\cong \pi^\sharp\circ \omega^\sharp\circ \pi^\sharp.\]
By nondegeneracy of $\pi^\sharp$ this is equivalent to $\pi^\sharp \cong (\omega^\sharp)^{-1}$. The inductive statement requires one to construct obstruction spaces.

\begin{defn}
Suppose
\begin{equation}
\xymatrix{
O \ar@/^1.5pc/^{p}[r] & X  \ar@<.5ex>^{s_1}[l] \ar@<-.5ex>_{s_2}[l]
}
\label{eq:vanishinglocus}
\end{equation}
is a diagram of simplicial sets together with homotopies $p\circ s_1\cong \id_X$ and $p\circ s_2\cong \id_X$. Its \defterm{vanishing locus} is the limit of
\[
\xymatrix{
X & O \ar_{p}[l] & X  \ar@<.5ex>^{s_1}[l] \ar@<-.5ex>_{s_2}[l]
}
\]
\end{defn}

In other words, the vanishing locus of the diagram \eqref{eq:vanishinglocus} consists of a point $x\in X$ and a homotopy $s_1(x)\sim s_2(x)$ in $p^{-1}(x)$. It is called a vanishing locus since the section $s_2$ in examples will be a ``zero section'' in some sense.

Studying the problem of lifting a Poisson structure $\pi\in\Pois^{\leq m}(A, n)$ to $\Pois^{\leq(m+1)}(A, n)$ one naturally encounters obstructions. Indeed, suppose
\[\pi = \pi_2 + \dots + \pi_m\in\Pois^{\leq m}(A, n).\]
Its lift
\[\pi' = \pi_2 + \dots + \pi_m + \pi_{m+1}\in\Pois^{\leq(m+1)}(A, n)\]
satisfies the Maurer--Cartan equation
\[\d \pi' + \frac{1}{2}[\pi', \pi'] = 0\]
in the truncation $\Pol^{\leq (m+1)}(A, n)$. Since $\pi$ satisfies the Maurer--Cartan equation up to weight $m$, we need to check the Maurer--Cartan equation only in weight $(m+1)$ which gives
\[\d \pi_{m+1} + \frac{1}{2}\sum_{i+j=m+2}[\pi_i, \pi_j] = 0.\]

Thus, the fiber of $\Pois^{\leq(m+1)}(A, n)\rightarrow \Pois^{\leq m}(A, n)$ at $\pi$ consists of null-homotopies of
\[\obs(\pi) = \frac{1}{2}\sum_{i+j=m+2}[\pi_i, \pi_j] \in\Pol^{m+1}(A, n).\]

Denote by $\Obs(\Pois, m+1)\rightarrow \Pois^{\leq m}(A, n)$ the trivial local system over $\Pois^{\leq m}(A, n)$ with fiber $\Pol^{m+1}(A, n)$. Then we have two sections: the zero section and $\obs(-)$. Thus, we obtain an obstruction diagram as in \eqref{eq:vanishinglocus}. One similarly constructs obstruction diagrams
\[
\xymatrix{
\Obs(\Symp, m+1) \ar@/^1.5pc/[r] & \Symp^{\leq m}(A, n) \ar@<.5ex>^-{\obs}[l] \ar@<-.5ex>_-{0}[l]
}
\]
and
\[
\xymatrix{
\Obs(\Comp, m+1) \ar@/^1.5pc/[r] & \Comp^{\leq m}(A, n) \ar@<.5ex>^-{\obs}[l] \ar@<-.5ex>_-{0}[l]
}
\]

These obstruction diagrams restrict to obstruction diagrams for nondegenerate pairs. Then one proves the following key lemmas.

\begin{lm}
The vanishing locus of
\[
\xymatrix{
\Obs(\Comp, m+1) \ar@/^1.5pc/[r] & \Comp^{\leq m}(A, n) \ar@<.5ex>[l] \ar@<-.5ex>[l]
}
\]
is equivalent to
\[\Comp^{\leq(m+1)}(A, n)\longrightarrow \Comp^{\leq m}(A, n)\]
and similarly for $\Symp$ and $\Pois$.
\end{lm}

\begin{lm}
The square
\[
\xymatrix{
\Obs(\Comp, m+1) \ar[r] \ar[d] & \Comp^{nd, \leq m}(A, n) \ar[d] \\
\Obs(\Symp, m+1) \ar[r] & \Symp^{\leq m}(A, n)
}
\]
is Cartesian.
\end{lm}

This proves the inductive step and hence \cref{thm:nondegeneratepoisson}.

\ 

One also has a relative analog of \cref{thm:nondegeneratepoisson}. Given an $n$-shifted coisotropic structure $(\pi, \gamma)$ on $f\colon L\rightarrow X$, where $\pi$ is an $n$-shifted Poisson structure on $X$, one has an induced morphism
\[\gamma^\sharp\colon \bL_{L/X}\longrightarrow \bT_L[-n].\]

\begin{defn}
Let $f\colon L\rightarrow X$ be a morphism of derived Artin stacks locally of finite presentation.
\begin{itemize}
\item Suppose $(\pi, \gamma)$ is an $n$-shifted coisotropic structure on $f$. It is \defterm{nondegenerate} if the $n$-shifted Poisson structure $\pi$ on $X$ is nondegenerate and $\gamma^\sharp\colon \bL_{L/X}\rightarrow \bT_L[-n]$ is an equivalence.

\item $\Cois^{nd}(f, n)\subset \Cois(f, n)$ is the subspace of nondegenerate $n$-shifted coisotropic structures.
\end{itemize}
\end{defn}

The following result was proved in \cite{Pri4} for $n=0$ and \cite[Theorem 4.22]{MS2} for any $n$.

\begin{thm}
Suppose $f\colon L\rightarrow X$ is a morphism of derived Artin stacks locally of finite presentation. Then one has an equivalence of spaces
\[\Cois^{nd}(f, n)\cong \Lagr(f, n)\]
fitting into a diagram
\[
\xymatrix{
\Cois^{nd}(f, n) \ar^{\sim}[r] \ar[d] & \Lagr(f, n) \ar[d] \\
\Pois^{nd}(X, n) \ar^{\sim}[r] & \Symp(X, n).
}
\]

Under this equivalence the diagram
\[
\xymatrix{
\bL_{L/X}[-1] \ar[r] \ar^{\gamma^\sharp}_{\sim}[d] & f^*\bL_X \ar[r] \ar^{\pi^\sharp}_{\sim}[d] & \bL_L \ar^{\gamma^\sharp}_{\sim}[d] \\
\bT_L[-n] \ar[r] & f^*\bT_X[-n] \ar[r] & \bT_{L/X}[1-n]
}
\]
induced by a nondegenerate $n$-shifted coisotropic structure $(\pi, \gamma)$ is inverse to the diagram
\[
\xymatrix{
\bT_L[-n] \ar[r] \ar^{\lambda^\sharp}_{\sim}[d] & f^*\bT_X[-n] \ar[r] \ar^{\omega^\sharp}_{\sim}[d] & \bT_{L/X}[1-n] \ar^{\lambda^\sharp}_{\sim}[d] \\
\bL_{L/X}[-1] \ar[r] & f^*\bL_X \ar[r] & \bL_L
}
\]
induced by the corresponding $n$-shifted Lagrangian structure $(\omega, \lambda)$.
\end{thm}

\begin{cor}
Under the previous assumptions we have a forgetful map
\[\Lagr(f, n)\cong \Cois^{nd}(f, n)\longrightarrow \Pois(L, n-1).\]

Given a Lagrangian structure $(\omega, \lambda)$ on $f$, the induced $(n-1)$-shifted Poisson structure on $L$ has bivector given by the composite
\[\bL_L\longrightarrow \bL_{L/X}\xrightarrow{(\lambda^\sharp)^{-1}} \bT_L[1-n].\]
\label{cor:lagrangianpoisson}
\end{cor}

\begin{example}
Suppose $f\colon L\rightarrow X$ is a smooth Lagrangian subscheme of a $0$-shifted symplectic scheme $X$. We have an exact sequence
\[0\longrightarrow \T_L\longrightarrow f^* \T_X\longrightarrow \N_{L/X}\longrightarrow 0\]
of vector bundles on $L$ which gives rise to an extension class $\H^1(L, \N^*_{L/X}\otimes \T_L)$. Since $L$ is Lagrangian, we have an isomorphism $\N^*_{L/X}\cong \T_L$. Therefore, we obtain an element
\[\pi_L\in \H^1(L, \T_L\otimes \T_L)\]
which can be shown to be symmetric. We see that $\pi_L$, the bivector underlying the $(-1)$-shifted Poisson structure on $L$, measures the obstruction for the normal bundle to $L$ being split.
\label{ex:lagrangiannormalbundle}
\end{example}

Let $\g$ be a complex equipped with a nondegenerate pairing of degree $n$. Then $\widehat{\Sym}(\g^*[-1])$ becomes a $\bP_{n+3}$-algebra. A cyclic $L_\infty$ structure on $\g$ can be encoded in a potential $h\in\widehat{\Sym}(\g^*[-1])[n+3]$ which is a Hamiltonian for the vector field defining the $L_\infty$ structure, see e.g. \cite[Section 3]{BL}.

Suppose $(X, \omega)$ is a smooth symplectic scheme. Kapranov \cite{Kap} has defined a cyclic $L_\infty$ structure on $\T_X[-1]$ with a pairing of degree $-2$ given by $\omega$ whose potential is an element of $\Gamma(X, \widehat{\Sym}^{\geq 3}(\T_X))$ of cohomological degree $1$.

\begin{conjecture}
Suppose $(X, \omega)$ is a smooth symplectic scheme. The $(-1)$-shifted Poisson structure on $X$ obtained as the image of $\omega$ under the forgetful map $\Pois(X, 0)\rightarrow \Pois(X, -1)$ gives the potential for the cyclic $L_\infty$ structure on $\T_X[-1]$.
\label{conj:atiyahbracket}
\end{conjecture}

For instance, we claim that the bivector underlying the $(-1)$-shited Poisson structure on $X$ is trivial. It can be seen as follows. The diagonal $X\rightarrow X\times X$ is Lagrangian and the image of $\omega$ under $\Pois(X, 0)\rightarrow \Pois(X, -1)$ coincides with the image of the Lagrangian structure on the diagonal under
\[\Lagr(X\rightarrow X\times X, 0)\rightarrow \Pois(X, -1).\]

The normal bundle to the diagonal embedding is split, so by \cref{ex:lagrangiannormalbundle} the bivector underlying the $(-1)$-shifted Poisson structure is trivial.

\section{Examples}

In the last lecture we give examples of shifted Poisson and shifted symplectic structures.

\subsection{Symplectic realizations}

Recall from \cref{cor:lagrangianpoisson} that given a morphism $f\colon L\rightarrow X$ of derived Artin stacks locally of finite presentation one has a forgetful map
\begin{equation}
\Lagr(f, n)\longrightarrow \Pois(L, n-1).
\label{eq:lagrangianpoisson}
\end{equation}

\begin{defn}
Let $L$ be a derived Artin stack equipped with an $(n-1)$-shifted Poisson structure. A \defterm{symplectic realization of $L$} is the data of an $n$-shifted Lagrangian $f\colon L\rightarrow X$ which lifts the $(n-1)$-shifted Poisson structure on $L$ along \eqref{eq:lagrangianpoisson}.
\end{defn}

The following statement is currently being investigated by Costello--Rozenblyum and Calaque--Vezzosi.

\begin{conjecture}
For a shifted Poisson stack $L$ there is a unique symplectic realization $L\rightarrow X$ for which $L\rightarrow X$ is a nil-isomorphism, i.e. an isomorphism on reduced stacks.
\label{conj:sympliealgd}
\end{conjecture}

\begin{remark}
In fact, generalizing \cref{prop:symplecticliealgd} one can \emph{define} the space $\Pois(X, n)$ in terms of such formal symplectic realizations. Such a definition allows one to reduce questions about shifted Poisson structures to questions about shifted symplectic structures which are significantly easier to work with. For instance, the coisotropic intersection theorem and the Poisson version of the AKSZ construction (\cite[Theorem 2.5]{PTVV}) follow almost immediately from the corresponding statements for shifted symplectic structures.
\end{remark}

To motivate this notion, let us recall the following definition from \cite{CDW}.

\begin{defn}
Let $\cG\rightrightarrows X$ be a smooth algebraic groupoid. It is a \defterm{symplectic groupoid} if one is given a multiplicative symplectic structure $\omega$ on $\cG$ for which the unit section $X\rightarrow \cG$ is Lagrangian.
\end{defn}

\begin{lm}
Suppose $\cG\rightrightarrows X$ is a symplectic groupoid. Then its Lie algebroid $\cL$ is a symplectic Lie algebroid.
\end{lm}

In particular, given a symplectic groupoid $\cG$ over $X$ one has an induced Poisson structure on $X$ by \cref{prop:symplecticliealgd}.

\begin{prop}
Let $\cG\rightrightarrows X$ be a smooth algebraic groupoid over a smooth scheme $X$. Then the space $\Lagr(X\rightarrow [X/\cG], 1)$ is equivalent to the set of symplectic structures $\omega$ on $\cG$ endowing it with the structure of a symplectic groupoid.
\end{prop}

A symplectic groupoid $\cG\rightrightarrows X$ lifting a given Poisson structure on $X$ gives an example of a symplectic realization of $X$ in the sense of \cite{We}. Thus, the forgetful map
\[\Lagr(X\rightarrow [X/\cG], 1)\longrightarrow \Pois(X, 0)\]
gives a different interpretation of the underlying Poisson structure on $X$.

Let us also explain some geometric consequences one can extract from a symplectic realization. Suppose that $L\rightarrow X$ is a $1$-shifted Lagrangian such that $L$ is a smooth Poisson manifold. The Lagrangian structure gives an isomorphism
\[\lambda^\sharp\colon \T_{L/X}\xrightarrow{\sim} \T^*_L.\]

By \cref{cor:lagrangianpoisson} the underlying Poisson structure is given by the composite
\[\T^*_L\xrightarrow{(\lambda^\sharp)^{-1}} \T_{L/X}\rightarrow \T_L.\] The image of the Poisson bivector under $\T^*_L\rightarrow \T_L$ gives a foliation of $L$ by the symplectic leaves of the Poisson structure which we see is the same as the tangent foliation $\T_{L/X}\rightarrow \T_L$. More generally, given a symplectic realization $L\rightarrow X$ of an $n$-shifted Poisson structure on $L$ we can think of it as follows:
\begin{itemize}
\item $X$ is the moduli space of symplectic leaves of the $n$-shifted Poisson structure on $L$.

\item The fibers of $L\rightarrow X$ are the symplectic leaves.
\end{itemize}

\begin{example}
Suppose $\g$ is a finite-dimensional Lie algebra of an algebraic group $G$ and consider the Poisson scheme $\g^*$ equipped with Kirillov--Kostant Poisson structure. It is well-known that the symplectic leaves of $\g^*$ are given by coadjoint orbits. In fact, it admits a symplectic realization
\[\g^*\longrightarrow [\g^*/G].\]
This 1-shifted Lagrangian can be constructed as follows. We can identify $[\g^*/G]\cong \T^*[1](\B G)$ which endows it with a 1-shifted symplectic structure. The above Lagrangian can be identified with the 1-shifted conormal bundle of the basepoint $\pt\rightarrow \B G$ which carries a 1-shifted Lagrangian structure as in \cref{ex:conormal}.

The corresponding symplectic groupoid can be extracted as
\[\g^*\times_{[\g^*/G]} \g^*\cong \T^* G\rightrightarrows \g^*\]
which coincides with the standard symplectic groupoid integrating the Poisson structure on $\g^*$.
\end{example}

\subsection{Classifying stacks}

Given a group $G$ and a polyvector $p\in \wedge^k(\g)$ at the unit of $G$, we can extend it using left or right translations to elements we denote by $p^\L, p^\R\in\wedge^k(\T_G)$.

\begin{defn}
Let $G$ be an affine algebraic group. A \defterm{quasi-Poisson structure} on $G$ is the data of a bivector $\pi\in\wedge^2 \T_G$, a trivector $\phi\in\wedge^3(\g)$ satisfying
\begin{align*}
\pi(g_1 g_2) &= L_{g_1, *} \pi(g_2) + R_{g_2, *} \pi(g_1) \\
\frac{1}{2}[\pi, \pi] + \phi^\L - \phi^R &= 0 \\
[\pi, \phi^\L] &= 0.
\end{align*}
\label{def:quasipoisson}
\end{defn}

Given a quasi-Poisson structure $(\pi, \phi)$ on $G$ and a bivector $\lambda\in\wedge^2(\g)$ we can construct a new quasi-Poisson structure on $G$ via a procedure called twisting. We have the following interpretation of quasi-Poisson structures in terms of shifted Poisson structures (see \cite[Proposition 2.5]{Sa3} and \cite[Theorem 2.8]{Sa3} for details).

\begin{prop}
Let $G$ be an affine algebraic group.
\begin{enumerate}
\item The space $\Pois(\B G, 2)$ is equivalent to the set $\Sym^2(\g)^G$.

\item The space $\Pois(\B G, 1)$ is equivalent to the groupoid of quasi-Poisson structures on $G$ with morphisms given by twists.
\end{enumerate}
\label{prop:quasipoissonshiftedpoisson}
\end{prop}

Let us briefly explain how one computes the spaces $\Pois(\B G, n)$. Recall from \cref{ex:BG} that under the identification $\QCoh(\B G)\cong \Rep G$ we have $\bL_{\B G}\cong \g^*[-1]$. Moreover, the functor of global sections $\Gamma(\B G, -)$ corresponds to taking the group cochains $\C^\bullet(G, -)$ valued in a representation. The algebra of polyvectors is then quasi-isomorphic to
\[\Pol(\B G, n)\cong \C^\bullet(G, \Sym(\g[-n])).\]

By \cref{prop:poissonpolyvectors} we can compute $n$-shifted Poisson structures on $\B G$ in terms of Maurer--Cartan elements of weight $\geq 2$ in $\Pol(\B G, n)[n+1]$. Observe that the weight $p$ part of $\Pol(\B G, n)$ is concentrated in cohomological degrees $[pn, \infty)$ for $n\geq 0$. We get the following cases:
\begin{itemize}
\item For $n > 2$ all Maurer--Cartan elements are trivial for degree reasons since in this case $pn - n - 1 > 1$ for $p\geq 2$.

\item For $n = 2$ a Maurer--Cartan element has to be concentrated in weight $2$, i.e. it is given by an element $c\in\C^0(G, \Sym^2(\g))$. The Maurer--Cartan equation forces it to be a cocycle, i.e. it has to be $G$-invariant.

\item Finally, for $n=1$ a Maurer--Cartan element is given by a pair of elements $\pi\in\C^1(G, \wedge^2(\g))$ and $\phi\in\C^0(G, \wedge^3(\g))$. The Maurer--Cartan equation boils down to the quasi-Poisson relations in \cref{def:quasipoisson}.
\end{itemize}

\begin{defn}
Let $G$ be an affine algebraic group. A \defterm{Poisson-Lie structure} on $G$ is the same as a multiplicative Poisson structure $\pi$ on $G$. Equivalently, it is a quasi-Poisson structure $(\pi, \phi=0)$.
\end{defn}

The Lie algebra $\g$ of a Poisson-Lie group $G$ inherits a Lie bialgebra structure. Let us denote by $\hG$ the corresponding formal group. Then a Poisson-Lie structure on $\hG$ is the same as a Lie bialgebra structure on $\g$. The following statement is shown in \cite[Corollary 2.10]{Sa3}.

\begin{prop}
Let $G$ be an affine algebraic group.
\begin{enumerate}
\item The space $\Pois(\pt\rightarrow \B G, 1)$ is equivalent to the set of Poisson-Lie structures on $G$.

\item The space $\Pois(\pt\rightarrow \B\hG, 1)$ is equivalent to the set of Lie bialgebra structures on $\g$.
\end{enumerate}
\end{prop}

These statements essentially follow from the description of the space of 1-shifted Poisson structures on $\B G$ given by \cref{prop:quasipoissonshiftedpoisson}. Indeed, recall that for a Poisson manifold $X$ the inclusion of a point $\pt\rightarrow X$ is Poisson iff the Poisson structure vanishes at the point. Thus, $\Pois(\pt\rightarrow \B G, 1)$ consists of 1-shifted Poisson structures on $\B G$ which vanish at the basepoint. The restriction to the basepoint on the level of polyvector fields is given by
\[\Pol(\B G, 1)\cong \C^\bullet(G, \Sym(\g[-1]))\longrightarrow \C^0(G, \Sym(\g[-1])),\]
so all such 1-shifted Poisson structures have $\phi = 0$.

There is also an interesting interpretation of maps \eqref{eq:forgetshiftstack} given by forgetting the shift. Given an element $c\in\Sym^2(\g)^G\subset \g\otimes \g$ we denote by $c_{12}\in (\U \g)^{\otimes 3}$ the element $c\otimes 1\in\g\otimes \g\otimes \U\g$ and similarly for $c_{23}$. Then one can see that $[c_{12}, c_{23}]\in \g^{\otimes 3}$. More concretely, pick a basis $\{e_i\}$ of $\g$ with structure constants $f_{ij}^k$ so that
\[[e_i, e_j] = \sum_k f_{ij}^k e_k.\]
Let $c = \sum_{i,j} c^{ij} e_i\otimes e_j$. Then
\[[c_{12}, c_{23}] = \sum_{i,j,k,l,a} c^{ij} c^{kl} f^a_{jk} e_i\otimes e_a\otimes e_l.\]
The following statement combines \cite[Proposition 2.15]{Sa3} and \cite[Proposition 2.17]{Sa3}.

\begin{prop}
Let $G$ be an affine algebraic group.
\begin{enumerate}
\item The map $\Pois(\B G, 2)\rightarrow \Pois(\B G, 1)$ sends $c\in\Sym^2(\g)^G$ to a quasi-Poisson structure with $(\pi=0, \phi)$ where
\[\phi = -\frac{1}{6}[c_{12}, c_{23}]\in\wedge^3(\g)^G.\]

\item The image of a Lie bialgebra structure under $\Pois(\pt\rightarrow \B\hG, 1)\rightarrow \Pois(\pt\rightarrow \B\hG, 0)$ is trivial iff the formal Poisson-Lie group $G^*$ integrating $\g^*$ is formally linearizable, i.e. it is Poisson isomorphic to the completion of $\g^*$ at the origin.
\end{enumerate}
\label{prop:forgetshiftBG}
\end{prop}

Now suppose $H\subset G$ is an abelian subgroup equipped with a nondegenerate pairing. Given a function $r\colon H\rightarrow \g\otimes \g$, its differential $\ddr r$ defines a function $H\rightarrow \h\otimes \g\otimes \g$. We denote by $\Alt(\ddr r)$ the function $H\rightarrow \wedge^3(\g)$ given by its complete antisymmetrization.

\begin{defn}
Let $G$ be an affine algebraic group and $H\subset G$ a closed subgroup. A \defterm{quasi-triangular classical dynamical $r$-matrix with base $H$} is an $H$-equivariant function $r\colon H\rightarrow \g\otimes \g$ satisfying
\begin{enumerate}
\item $[r_{12}, r_{13}] + [r_{12}, r_{23}] + [r_{13}, r_{23}] + \Alt(\ddr r) = 0$.

\item The symmetric part of $r$ is a constant element of $\Sym^2(\g)^G$.
\end{enumerate}
\end{defn}

Note that the quotient $[H/H]$ has a canonical 1-shifted symplectic structure. Indeed, it can be written as a Lagrangian self-intersection of the diagonal
\[\B H\times_{\B H\times \B H} \B H,\]
where $\B H$ carries a 2-shifted symplectic structure, to which we can apply \cref{prop:lagrangianintersection}. Alternatively, we can identify
\[[H/H]\cong \Map(S^1_\B, \B H)\]
and use the AKSZ construction given by \cite[Theorem 2.5]{PTVV} to transgress the 2-shifted symplectic structure on $\B H$ to a 1-shifted symplectic structure on $[H/H]$. Finally, since $H$ is abelian, we can identify $[H/H]\cong H\times \B H$ and write in coordinates a 1-shifted symplectic structure which pairs $H$ and $\B H$.

Using the 1-shifted symplectic on $[H/H]$ we can turn it into a 1-shifted Poisson structure on $[H/H]$ using \cref{thm:nondegeneratepoisson}.

\begin{prop}
The set of quasi-triangular classical dynamical $r$-matrices with base $H$ is equivalent to the space of pairs:
\begin{enumerate}
\item A 1-shfited Poisson map
\[[H/H]\longrightarrow \B H\longrightarrow \B G\]
compatible with the given $1$-shifted Poisson structure on $[H/H]$.

\item A lift of the underlying 1-shifted Poisson structure on $\B G$ to a 2-shifted Poisson structure.
\end{enumerate}
\label{prop:dynamicalrmatrix}
\end{prop}

\subsection{Moduli of bundles}
\label{sect:modulibundles}

Let $G$ be a split simple group, let $B\subset G$ be the Borel subgroup and $H\subset B$ the Cartan subgroup. Choose a nondegenerate pairing in $\Sym^2(\g^*)^G\cong \Sym^2(\h^*)^W$ which defines 2-shifted symplectic structures on $\B G$ and $\B H$. The following is shown in \cite[Lemma 3.4]{Sa4}.

\begin{prop}
The correspondence
\begin{equation}
\xymatrix{
& \B B \ar[dl] \ar[dr] & \\
\B H && \B G
}
\label{eq:borelcorrespondence}
\end{equation}
has a unique 2-shifted Lagrangian structure.
\end{prop}

Let $E$ be an elliptic curve and denote by $\Bun_{(-)}(E) = \Map(E, \B(-))$ the moduli stack of principal bundles on $E$. Applying \cite[Theorem 2.5]{PTVV} we get a 1-shifted Lagrangian correspondence
\[
\xymatrix{
& \Bun_B(E) \ar[dl] \ar[dr] & \\
\Bun_H(E) && \Bun_G(E)
}
\]

In particular, we get a 0-shifted Poisson structure on $\Bun_B(E)$ known as the Feigin--Odesskii Poisson structure introduced in \cite{FO}, see \cite{HP} and \cite[Example 4.11]{Sa3} for more details.

\begin{example}
Let $G=\PGL_2$, so $H=\bG_m$. Fix a line bunde $\cL\in\Bun_H(E)$ and let $\Bun_B(E; \cL)$ be the fiber of $\Bun_B(E)\rightarrow \Bun_H(E)$ at $\cL$. Then by the above $\Bun_B(E;\cL)\rightarrow \Bun_G(E)$ has a 1-shifted Lagrangian structure.

The stack $\Bun_B(E; \cL)$ parametrizes rank $2$ vector bundles $\cF$ on $E$ which fit into an exact sequence
\[0\longrightarrow \cL\longrightarrow \cF\longrightarrow \cO\longrightarrow 0.\]

Such extensions are parametrized by
\[\H^1(E; \cL)\cong \H^0(E; \cL^{-1})^*.\]

Suppose $\deg(\cL) = -4$. Then
\[\Bun_B(E; \cL)\cong [\bA^4 / \bG_m].\]
It has an open substack given by $\bP^3$. Therefore, we get a Poisson structure on $\bP^3$ which is known as the Sklyanin Poisson structure (introduced in \cite{Sk}).
\end{example}

The AKSZ construction given by \cite[Theorem 2.5]{PTVV} shows that the functor $\Map(E, -)$ sends $n$-shifted symplectic stacks to $(n-1)$-shifted symplectic stacks. Spaide \cite{Sp} has given a version of this construction which shows that the moduli space of framed maps $\Map^{fr}(\bP^1, -)$ sends $n$-shifted symplectic stacks to $(n-1)$-shited symplectic stacks. Here $\Map^{fr}(\bP^1, -)$ consists of maps from $\bP^1$ which have a specified value at the basepoint of $\bP^1$.

In particular, we get a $1$-shifted Lagrangian correspondence
\[
\xymatrix{
& \Map^{fr}(\bP^1, \B B) \ar[dl] \ar[dr] & \\
\Map^{fr}(\bP^1, \B H) && \Map^{fr}(\bP^1, \B G).
}
\]

The stack $\Map^{fr}(\bP^1, \B H)$ is a discrete set, so restricting to a fixed component $\lambda\in\Map^{fr}(\bP^1, \B H)$, we get a 1-shifted Lagrangian
\[\Map^{fr,\lambda}(\bP^1, \B B)\longrightarrow \Map^{fr}(\bP^1, \B G).\]
One can check that the inclusion of the trivial $G$-bundle
\[\pt\longrightarrow \Map^{fr}(\bP^1, \B G)\]
is also a 1-shifted Lagrangian, so
\[\Map^{fr,\lambda}(\bP^1, G/B)\cong \Map^{fr, \lambda}(\bP^1, \B B)\times_{\Map^{fr}(\bP^1, \B G)} \pt\]
has a 0-shifted symplectic structure. The scheme $\Map^{fr, \lambda}(\bP^1, G/B)$ by a result of Jarvis \cite{Ja} (following some work by Donaldson) can be identified with the moduli space of $G$-monopoles on $\mathbf{R}^3$ with maximal symmetry breaking at infinity.

Finally, let us construct a classical dynamical $r$-matrix. Applying $\Map(S^1_\B, -)$ to the 2-shifted Lagrangian correspondence \eqref{eq:borelcorrespondence} by \cite[Theorem 2.5]{PTVV} we get a 1-shifted Lagrangian correspondence
\[
\xymatrix{
& [B/B] \ar[dl] \ar[dr] & \\
[H/H] && [G/G]
}
\]

The projection $[B/B]\rightarrow [H/H]$ is an isomorphism over the locus $[H^{reg}/H]\subset [H/H]$ where the Weyl group acts freely. For instance, in the case of $G=\mathrm{SL}_n$, $H^{reg}\subset H$ consists of diagonal matrices with distinct eigenvalues. Therefore, we get a 1-shifted Lagrangian correspondence
\[
\xymatrix{
& [H^{reg}/H] \ar[dl] \ar[dr] & \\
[H^{reg}/H] && [G/G].
}
\]
which is a graph of a map $[H^{reg}/H]\rightarrow [G/G]$ which is hence a symplectomorphism. The composite
\[[H^{reg}/H]\longrightarrow [G/G]\longrightarrow \B G\]
therefore has a 1-shifted Poisson structure. So, by \cref{prop:dynamicalrmatrix} we obtain a dynamical $r$-matrix. One can check that in this way we get the basic trigonometric dynamical $r$-matrix, see \cite[Proposition 3.3]{ES}.

\printbibliography

\end{document}